\documentclass{amsart}
\allowdisplaybreaks[4]

\newcommand{\Z}{\mathbf{Z}}
\newcommand{\R}{\mathbf{R}}
\newcommand{\C}{\mathbf{C}}

\newcommand{{\ba}}{\bf a}
\newcommand{\ve}{\varepsilon}
\newcommand{\la}{\lambda}
\newcommand{\La}{\Lambda}
\newcommand{\ga}{\gamma}
\newcommand{\pa}{\partial}
\newcommand{\ra}{\rightarrow}
\newcommand{\Om}{\Omega}
\newcommand{\del}{\delta}
\newcommand{\Del}{\Delta}

\newcommand{\na}{\nabla}
\newcommand{\cd}{\cdot}

\newcommand{\al}{\alpha}

\newcommand{\be}{\begin{equation}}
\newcommand{\ee}{\end{equation}}

\newcommand{\om}{\omega}

\newtheorem{lem}{Lemma}{\bf}{\it}
\newtheorem{remark}{Remark}{\it}{\rm}

\newtheorem{theorem}{Theorem}
\newtheorem{proposition}{Proposition}
\newtheorem{corollary}{Corollary}
\numberwithin{theorem}{section}
\numberwithin{lem}{section}
\numberwithin{equation}{section}
\numberwithin{proposition}{section}
\numberwithin{corollary}{section}
\title[Strong Convergence]{Strong Convergence to the Homogenized Limit of Elliptic Equations with Random Coefficients}

\author{Joseph G. Conlon and Thomas  Spencer}

\address{ (Joseph G. Conlon): University of Michigan, Department of Mathematics, Ann Arbor,
  MI 48109-1109}
\email{conlon@umich.edu}

\address{ (Thomas Spencer): School of Mathematics, Institute for Advanced Study, Princeton, NJ 08540}
\email{spencer@math.ias.edu}

\keywords{Euclidean field theory, pde with random coefficients, homogenization}
\subjclass{81T08, 82B20, 35R60, 60J75}

\begin{document}

\maketitle

\begin{abstract}
Consider a discrete uniformly elliptic divergence form equation on the $d$ dimensional lattice $\Z^d$ with random coefficients. It has previously been shown that if the random environment is translational invariant, then the averaged Green's function together with its first and second differences, are bounded by the corresponding quantities for the constant coefficient discrete elliptic equation. It has also been shown that if the random environment is ergodic, then solutions of the random equation converge under diffusive scaling  to solutions of a homogenized elliptic PDE on $\R^d$. In this paper point-wise estimates are obtained on the difference between the averaged Green's function and the homogenized Green's function for certain random environments which are strongly mixing.

\end{abstract}

\section{Introduction.}
Let $(\Om,\mathcal{F},P)$ be a probability space and  denote by $\langle \  \cd \ \rangle$ expectation w.r. to the measure $P$.   We assume that the $d$ dimensional integer lattice $\Z^d$ acts on $\Om$ by translation operators $\tau_x:\Om\ra\Om, \ x\in\Z^d$, which are measure preserving and satisfy the properties $\tau_x\tau_y=\tau_{x+y}, \ \tau_0= \ {\rm identity}, \ x,y\in\Z^d$. 
Consider a bounded measurable function  ${\bf a}:\Om\ra\R^{d(d+1)/2}$  from $\Om$ to the space of symmetric $d\times d$ matrices which satisfies the quadratic form inequality  
\begin{equation} \label{A1}
\la I_d \le {\bf a}(\om) \le \La I_d, \ \ \ \ \ \om\in\Om,
\end{equation}
where $I_d$ is the identity matrix in $d$ dimensions and $\La, \la$
are positive constants. We shall be interested in solutions $u(x,\eta,\om)$ to the discrete elliptic equation
\be \label{B1}
\eta u(x,\eta,\om)+\nabla^*{\bf a}(\tau_x\om)\nabla u(x,\eta,\om)=h(x), \quad x\in \Z^d, \ \om\in\Om.
\ee
In (\ref{B1}) we take $\eta\ge 0$ and $\nabla$ the discrete gradient operator, which has adjoint $\nabla^*$. Thus $\nabla$ is a $d$ dimensional { \it column} operator and $\nabla^*$ a $d$ dimensional {\it row} operator, which act on  functions $\phi:\Z^d\ra\R$ by
\begin{eqnarray} \label{C1}
\na \phi(x) &=& \big( \na_1 \phi(x),... \ \na_d\phi(x) \big), \quad  \na_i \phi(x) = \phi (x + {\bf e}_i) - \phi(x),  \\
\na^* \phi(x) &=& \big( \na^*_1 \phi(x),... \ \na^*_d\phi(x) \big), \quad  \na^*_i \phi(x) = \phi (x - {\bf e}_i) - \phi(x). \nonumber
\end{eqnarray}
In (\ref{C1}) the vector  ${\bf e}_i \in \Z^d$ has 1 as the ith coordinate and 0 for the other coordinates, $1\le i \le  d$. 

It is well known \cite{k,pv,zko} that if the translation operators $\tau_x, \ x\in\Z^d$, are ergodic  on $\Om$ then solutions to the random equation (\ref{B1}) converge to solutions of a constant coefficient equation under suitable scaling. Thus suppose $f:\R^d\ra\R$ is a $C^\infty$ function with compact support and for $\ve$ satisfying $0<\ve\le 1$ set $h(x)=\ve^2f(\ve x), \ x\in\Z^d$, in (\ref{B1}).  Then $u(x/\ve,\ve^2\eta,\om)$ converges with probability $1$ as $\ve\ra 0$ to a function $u(x,\eta), \ x\in\R^d$, which is the solution to the constant coefficient elliptic PDE
\be \label{D1}
\eta u(x,\eta)+\na^* {\bf a}_{\rm hom}\na u(x,\eta)=f(x), \quad x\in\R^d,
\ee
where the $d\times d$ symmetric matrix ${\bf a}_{\rm hom}$ satisfies the quadratic form inequality (\ref{A1}).  This homogenization result can be viewed as a kind of central limit theorem, and our purpose here will be to show that the theorem can be strengthened for certain probability spaces $(\Om,\mathcal{F},P)$.

We consider what the homogenization result says about the expectation of the Green's function for  equation (\ref{B1}). By translation invariance of the measure we have that
\be \label{E1}
\langle \  u(x,\eta,\cdot) \ \rangle=\sum_{y\in\Z^d} G_{{\bf a},\eta}(x-y)h(y), \quad x\in\Z^d,
\ee
where  $G_{{\bf a},\eta}(x)$ is the expected value of the Green's function. Setting $h(x)=\ve^2f(\ve x), \ x\in\Z^d$, then (\ref{E1}) may be written as
\be \label{F1}
\langle \ u(x/\ve,\ve^2\eta,\cdot) \ \rangle=\int_{\ve Z^d} \ve^{2-d}G_{{\bf a},\ve^2\eta}\left(\frac{x-z}{\ve}\right) f(z)  \ dz, \quad x\in \ve \Z^d,
\ee
where integration over $\ve\Z^d$ is defined by
\be \label{G1}
\int_{\ve Z^d} g(z) \ dz \ =  \ \sum_{z\in\ve\Z^d} g(z) \ \ve^d.
\ee
Let $G_{{\bf a}_{\rm hom},\eta}(x), \ x\in\R^d$, be the Greens function for the PDE (\ref{D1}). One easily sees that $G_{{\bf a}_{\rm hom},\eta}(\cdot)$ satisfies the scaling property
\be \label{H1}
\ve^{2-d} G_{{\bf a}_{\rm hom},\ve^2\eta}(x/\ve) \ = \  G_{{\bf a}_{\rm hom},\eta}(x), \quad \ve,\eta>0, \ x\in\R^d-\{0\}.
\ee
From (\ref{F1}), (\ref{H1}) we see that homogenization implies that the function $\ve^{2-d} G_{{\bf a},\ve^2\eta}(x/\ve), \ x\in\ve\Z^d$, converges in an averaged sense to the Greens function $G_{{\bf a}_{\rm hom},\eta}(x), \ x\in\R^d$. A consequence of our results here will be that for certain probability spaces $(\Om,\mathcal{F},P)$ and functions 
${\bf a}:\Om\ra\R^{d(d+1)/2}$  this convergence is {\it point-wise} in $x$. In particular for some $\alpha$ satisfying $0<\alpha\le 1$,  there are positive constants $C,\gamma$ such that
\be \label{I1}
|\ve^{2-d} G_{{\bf a},\ve^2\eta}(x/\ve)-G_{{\bf a}_{\rm hom},\eta}(x)| \ \le \ \frac{C\ve^\alpha}{[|x|+\ve]^{d-2+\alpha}} e^{-\gamma\sqrt{\eta/\La} |x|}, \quad 0<\ve\le1,  \ x\in\ve\Z^d-\{0\}.
\ee

We shall also denote by $G_{{\bf a}_{\rm hom},\eta}(x), \ x\in\Z^d$,  the Greens function for the difference equation (\ref{D1}) on $\Z^d$. Evidently the $\Z^d$ Green's function has the property that $G_{{\bf a}_{\rm hom},\eta}(0)$ is finite, unlike the corresponding $\R^d$ Green's function. It is also clear that the inequality (\ref{I1}) for $\ve<1$ follows from the same inequality for $\ve=1$ i.e.
\be \label{J1}
|G_{{\bf a},\eta}(x)-G_{{\bf a}_{\rm hom},\eta}(x)|  \le  \frac{C}{\La(|x|+1)^{d-2+\alpha}} e^{-\gamma\sqrt{\eta/\La} |x|}, \    \ x\in\Z^d,  
\ee
provided we are able to obtain an inequality (\ref{J1}) which is uniform in $\eta>0$ as $\eta\ra 0$. 
We shall prove such an inequality and also  similar inequalities for the derivatives of the expectation of the Green's function,
\begin{eqnarray} \label{K1}
 |\na G_{{\bf a},\eta}(x)-\na G_{{\bf a}_{\rm hom},\eta}(x)|  &\le& \frac{C}{\La(|x|+1)^{d-1+\alpha}} e^{-\gamma\sqrt{\eta/\La} |x|},   \\  
 |\na\na G_{{\bf a},\eta}(x)-\na\na G_{{\bf a}_{\rm hom},\eta}(x)|  &\le& \frac{C}{\La(|x|+1)^{d+\alpha}} e^{-\gamma\sqrt{\eta/\La} |x|}. \label{M1}
\end{eqnarray}
\begin{theorem} Suppose ${\bf a}(\cdot)$ satisfies (\ref{A1}),  the matrices ${\bf a}(\tau_x\cdot), \ x\in\Z^d$, are independent, and $0<\eta\le\La$. Then for $d\ge 2$ there exists $\alpha>0$ depending only on $d$ and  $\La/\la$, such that (\ref{J1}), (\ref{K1}) and (\ref{M1}) hold for some positive constants $\ga, C$, depending only on $d$ and 
$\La/\la$. 
\end{theorem}
We also consider here probability spaces $(\Om,\mathcal{F},P)$ corresponding to certain Euclidean field theories. These Euclidean field theories are determined by a potential $V : \R^d \ra \R$ which is a $C^2$ uniformly convex function.  Thus the second derivative ${\bf a}(\cdot)=V''(\cdot)$ of $V(\cdot)$ is assumed to satisfy the inequality (\ref{A1}). Next consider functions 
$\phi : \Z^d \ra \R$ on the integer lattice in $\R^d$.  Let $\Om$ be the space of
all such functions and $\mathcal{F}$ be the Borel algebra generated by finite 
dimensional rectangles 
$\{ \phi \in \Om: \  |\phi(x_i) - a_i| < r_i, \ i=1,...,N\}$, 
$x_i \in \Z^d, \ a_i \in \R, \ r_i > 0, \ i=1,...,N, \ N \ge 1$.  The translation operators $\tau_x:\Om\ra\Om, \ x\in\Z^d$, are defined by $\tau_x\phi(z)=\phi(x+z), \ z\in\Z^d$. For any $d\ge 1$ and $m>0$ one can define \cite{c1,fs} a unique ergodic translation invariant probability 
measure $P$ on $(\Om, \mathcal{F})$ which depends on the function $V$ and $m$.  The 
measure is formally given as
\begin{equation} \label{L1}
\exp \left[ - \sum_{x\in \Z^d} V\left( \na\phi(x)\right)+m^2\phi(x)^2 \right] \prod_{x\in \Z^d} d\phi(x)/{\rm normalization}.
\end{equation}
\begin{theorem} Let $\tilde{{\bf a}}:\R\ra\R^{d(d+1)/2}$  be a $C^1$ function on $\R$ with values in the space of symmetric $d\times d$ matrices which satisfy the quadratic form inequality  (\ref{A1}). Let $(\Om, \mathcal{F}, P)$ be the probability space of  fields $\phi(\cdot)$ determined by (\ref{L1}), and set ${\bf a}(\cdot)$ in (\ref{B1}) to be ${\bf a}(\phi)=\tilde{{\bf a}}(\phi(0)), \ \phi\in\Om$.  
Suppose in addition that the derivative $D\tilde{{\bf a}}(\cdot)$ of $\tilde{{\bf a}}(\cdot)$ satisfies  the inequality $\|D\tilde{{\bf a}}(\cdot)\|_{\infty}\le \La_1$.  Then for $d\ge 2$ there exists $\alpha>0$ depending only on $d$ and  $\La/\la$, such that (\ref{J1}), (\ref{K1}) and (\ref{M1}) hold for some positive constants $\ga$ and $C=C_1[\La_1/m\La+1]$, where $\ga,C_1$ depend only on $d$ and $\La/\la$. 
\end{theorem}

The limit as $m\ra 0$ of the measure (\ref{L1}) is a probability measure on gradient fields $\om:\Z^d\ra\R^d$, where formally $\om(x)=\na(\phi(x)), \ x\in\Z^d$. This massless field theory measure is ergodic with respect to translation operators \cite{c1,fs} for all $d\ge 1$. In the case $d=1$ it has a simple structure since  then the variables $\om(x), \ x\in\Z$, are i.i.d.  Note that in the probability space $(\Om,\mathcal{F},P)$ for the massless field theory,  the Borel algebra $\mathcal{F}$ is  generated by the intersection of finite dimensional rectangles and the hyperplanes imposing the gradient constraints for $\om(\cdot)$.  For $d\ge 3$ the gradient field theory measure induces a measure on fields $\phi:\Z^d\ra\R$ which is simply the limit of the measures (\ref{L1})  as $m\ra 0$.   For $d=1,2$ the $m\ra 0$ limit of the measures (\ref{L1}) on fields  $\phi:\Z^d\ra\R$ does not exist.  

We can show that the inequalities (\ref{J1}), (\ref{K1}), (\ref{M1}) also hold when $(\Om,\mathcal{F},P)$ is given by the massless field theory environment.
\begin{theorem} Let $\tilde{{\bf a}}:\R^d\ra\R^{d(d+1)/2}$  be a $C^1$ function on $\R^d$ with values in the space of symmetric $d\times d$ matrices which satisfy the quadratic form inequality  (\ref{A1}). Let $(\Om, \mathcal{F}, P)$ be the probability space of  gradient fields $\om(\cdot)=\na\phi(\cdot)$ determined by the limit of (\ref{L1}) as $m\ra 0$, and set ${\bf a}(\cdot)$ in (\ref{B1}) to be ${\bf a}(\om)=\tilde{{\bf a}}(\om(0)), \ \om\in\Om$.  
Suppose in addition that the derivative $D\tilde{{\bf a}}(\cdot)$ of $\tilde{{\bf a}}(\cdot)$ satisfies  the inequality $\|D\tilde{{\bf a}}(\cdot)\|_{\infty}\le \La_1$.  Then for $d\ge 2$ there exists $\alpha>0$ depending only on $d$ and  $\La/\la$, such that (\ref{J1}), (\ref{K1}) and (\ref{M1}) hold for some positive constants $\ga$ and $C=C_1[\La_1/\La\sqrt{\la}+1]$, where $\ga,C_1$ depend only on $d$ and $\La/\la$. 
\end{theorem}

Our method of proof for Theorems 1.1-1.3 combine methods used to prove regularity of averaged Green's functions for pde with random coefficients with methods for obtaining rates of convergence in homogenization. Regularity of averaged Green's functions was first proved in \cite{cn2}. The results of that paper imply that for any probability space $(\Om,\mathcal{F},P)$  with translation invariant operators $\tau_x:\Om\ra\Om, \ x\in\Z^d$, the inequalities (\ref{J1}), (\ref{K1}) hold for $\al=0$ and (\ref{M1}) for any $\al<0$. The approach of the paper is to obtain good control on the Fourier transform  
 $\hat{G}_{{\bf a},\eta}(\xi), \ \xi\in[-\pi,\pi]^d$, of  $G_{{\bf a},\eta}(x), \ x\in\Z^d$, for $|\xi|$ close to $0$. 
 Using the Fourier inversion formula, one then obtains the inequalities (\ref{J1})-(\ref{M1}). In \cite{dd} the inequality (\ref{M1}) is proven with $\al=0$, and in fact H\"{o}lder continuity of the second difference of
 $G_{{\bf a},\eta}(x), \ x\in\Z^d$, is also established. In contrast to \cite{cn2}, the approach of \cite{dd} is local in configuration space, and uses results from harmonic analysis which are deeper than those used in \cite{cn2}. In particular, the Harnack inequality \cite{gt} for uniformly elliptic equations in divergence form is needed to prove (\ref{M1}) with $\al=0$, whereas the proof of (\ref{M1}) with $\al<0$ in \cite{cn2} follows from interpolation inequalities.
 
 We have already observed that the inequality (\ref{J1}) with $\al>0$  implies (\ref{I1}), which gives a rate of convergence of $\ve^{\al}$ in homogenization. The first results establishing a rate of convergence for homogenization of elliptic PDE in divergence form were obtained in \cite{y}. These results require much stronger assumptions on the translation operators $\tau_x:\Om\ra\Om, \ x\in\Z^d,$ than ergodicity. One needs to assume that the variables ${\bf a}(\tau_x\cdot), \ x\in\Z^d,$ are independent, or at least are very weakly correlated. Rates of convergence in homogenization when $(\Om,\mathcal{F},P)$  is either the massive field theory of Theorem 1.2 or the massless field theory of Theorem 1.3 were obtained in \cite{ns2}. The main tools used to prove these results are the Brascamp-Lieb (BL) inequality \cite{bl} and Meyer's theorem \cite{m}. Meyer's theorem  is a consequence of the continuity in $p$ of the norms of Calderon-Zygmund operators acting on the spaces  $L^p(\Z^d)$ of functions whose $p$th powers are summable. In \cite{cn1} it was shown that Meyer's theorem could also be used to obtain a rate of convergence in the independent variable  case of Theorem 1.1. Recently \cite{go1,go2}  the independent variable case was taken up again, showing that the method of \cite{ns2}, which uses a BL inequality plus Meyer's theorem, could also be directly implemented in this case. Because of the perturbative nature of Meyer's theorem, this method alone does not yield optimal rates of convergence to homogenization.  However by combining  the method  with some deterministic estimates on Green's functions, optimal rates of convergence to homogenization are obtained in \cite{go1,go2} .
 
 In the present paper we follow the methodology of \cite{cn2} to obtain estimates on $\hat{G}_{{\bf a},\eta}(\xi), \ \xi\in[-\pi,\pi]^d$, which imply (\ref{J1}), (\ref{K1}) with $\al=0$. The estimates on  $\hat{G}_{{\bf a},\eta}(\xi)$ are improved by the use of Meyer's theorem for the independent variable environment, and by the BL inequality plus Meyer's theorem in the field theory case. The inequalities (\ref{J1}), (\ref{K1}) for some $\al>0$ follow then upon using the Fourier inversion formula. To prove (\ref{M1}) we also have to use the results of \cite{dd} to estimate the contribution of high Fourier modes. These estimates are  a consequence of the  H\"{o}lder continuity of the second difference of $G_{{\bf a},\eta}(x), \ x\in\Z^d$, already mentioned.

\section{Fourier Space Representation and Estimates}
In this section we summarize relevant results of previous work \cite{cn1,cn2} which were used to prove pointwise bounds on the Green's function $G_{{\bf a},\eta}(x), \ x\in\Z^d$, defined by (\ref{E1}). The starting point for this is the Fourier representation
\be \label{A2}
G_{{\bf a},\eta}(x) \ = \ 
\frac{1}{(2\pi)^d}\int_{[-\pi,\pi]^d} 
\frac{e^{-i\xi.x}}{\eta+e(\xi)^*q(\xi,\eta)e(\xi)} \ d\xi \ ,
\ee
where the $d\times d$ matrix function $q(\xi,\eta), \ \xi\in \R^d, \ \eta>0$, is a complex Hermitian  positive definite function of $(\xi,\eta)$, periodic in $\xi$ with fundamental region $[-\pi,\pi]^d$, which  satisfies the quadratic form inequality
\begin{equation} \label{B2}
\la I_d \le q(\xi,\eta) \le \La I_d, \quad \xi\in \R^d, \ \eta>0.
\end{equation}
The $d$ dimensional column vector $e(\xi)$ in (\ref{A2}) has $j$th entry $e_j(\xi)=e^{-i{\bf e}_j\cdot\xi}-1, \  1\le j\le d$. 

The function $q(\cdot,\cdot)$ is given in terms of the solution of an elliptic equation on $\Om$. For a measurable function  $\psi:\Om\ra\C$ we define the $\xi$ derivative of $\psi(\cdot)$ in the $j$ direction $\pa_{j,\xi}$, and its adjoint  $\pa^*_{j,\xi}$, by
\begin{eqnarray} \label{C2}
\pa_{j,\xi} \psi(\om) \ &=& \  e^{-i{\bf e}_j.\xi}\psi(\tau_{{\bf e}_j}\om)-\psi(\om),  \\
\pa^*_{j,\xi} \psi(\om) \ &=& \   e^{i{\bf e}_j.\xi}\psi(\tau_{-{\bf e}_j}\om)-\psi(\om). \nonumber
\end{eqnarray}
We also define a $d$ dimensional column $\xi$ derivative operator  $\pa_\xi$ by  $\pa_\xi=(\pa_{1,\xi},....,\pa_{d,\xi})$, which has adjoint  $\pa_\xi^*$  given by the row operator $\pa_\xi^*=(\pa^*_{1,\xi},....,\pa^*_{d,\xi})$. Let $\Phi(\xi,\eta,\om)$ be the d dimensional row vector which is the solution to the equation
\be \label{D2}
\eta\Phi(\xi,\eta,\om)+P\pa_\xi^*{\bf a}(\om)\pa_\xi\Phi(\xi,\eta,\om)=-P\pa^*_\xi {\bf a}(\om), \quad \eta>0, \ \xi\in \R^d, \ \om\in\Om,
\ee
where $P$ is the projection orthogonal to the constant function.
Then $q(\xi,\eta)$ is given in terms of the solution to (\ref{D2}) by the formula
\be \label{E2}
q(\xi,\eta) \ = \  \langle \ {\bf a}(\cdot) \ \rangle+ \langle \ {\bf a}(\cdot)\pa_\xi  \Phi(\xi,\eta,\cdot) \ \rangle \ .
\ee

The solution to (\ref{D2}) can be generated by a convergent perturbation expansion. Let $\mathcal{H}(\Om)$ be the Hilbert space of measurable functions $\psi:\Om\ra\C^d$ with norm $\|\psi\|$ given by $\|\psi\|^2=\langle \ |\psi(\cdot)|^2 \ \rangle$. We define an operator $T_{\xi,\eta}$  on  $\mathcal{H}(\Om)$  as follows: For any $g\in \mathcal{H}$, let 
$\psi(\xi,\eta,\om)$ be  the solution to the equation
\be \label{F2}
\frac{\eta}{\La}\psi(\xi,\eta,\om)+\pa_\xi^*\pa_\xi\psi(\xi,\eta,\om)=\pa^*_\xi g(\om), \quad \eta>0, \ \xi\in \R^d, \ \om\in\Om.
\ee
Then $T_{\xi,\eta} g(\cdot)=\pa_\xi\psi(\xi,\eta,\cdot)$, or more explicitly
\be \label{G2}
T_{\xi,\eta} g(\om) \ = \   \sum_{x\in \Z^d} \left\{\nabla\nabla^* G_{\eta/\La}(x)\right\}^*\exp[-ix.\xi]  \ g(\tau_x\om),
\ee 
where $G_\eta(\cdot)$ is the solution to the equation
\be \label{H2}
\eta G_\eta(x)+\nabla^*\nabla G_\eta(x)=\delta(x), \quad x\in\Z^d.
\ee
It is easy to see from (\ref{F2}) that $T_{\xi,\eta} $ is a bounded self-adjoint operator on $\mathcal{H}(\Om)$  with $\|T_{\xi,\eta}\|\le 1$, provided $\xi\in\R^d, \ \eta>0$. 
Now on setting ${\bf a}(\cdot)=\La[I_d-{\bf b}(\cdot)]$, one sees  that (\ref{D2}) is equivalent to the equation
\be \label{I2}
\pa_\xi\Phi(\xi,\eta,\cdot)=PT_{\xi,\eta}[{\bf b}(\cdot)\pa_\xi\Phi(\xi,\eta,\cdot)]
+PT_{\xi,\eta}[{\bf b}(\cdot)] \ .
\ee
Since  $\|T_{\xi,\eta}\|\le 1$ and $\|{\bf b}(\om)\|\le 1-\la/\La, \ \om\in\Om$,  the Neumann series for the solution to (\ref{I2}) converges in $\mathcal{H}(\Om)$. 

It will be useful later to express the operator $T_{\xi,\eta}$ in its Fourier representation. To do this we use the standard notation for the Fourier transform $\hat{h}(\zeta), \ \zeta\in[-\pi,\pi]^d$, of a function $h:\Z^d\ra\C$. Thus 
\be \label{Z2}
\hat{h}(\zeta) \ = \ \sum_{x\in\Z^d} h(x)e^{ix.\zeta}, \quad \zeta\in\R^d,
\ee
and the Fourier inversion formula yields
\be \label{AA2}
h(x) \ = \  \frac{1}{(2\pi)^d}\int_{[-\pi,\pi]^d} \hat{h}(\zeta)e^{-ix.\zeta} \ d\zeta, \quad x\in\Z^d \ .
\ee
Now the action of the translation group $\tau_x, \ x\in\Z^d$, on $\Om$ can be described by a set $A_1,...,A_d$ of commuting self-adjoint operators on $L^2(\Om)$, so that
\be \label{AB2}
f(\tau_x\cdot) \ = \ \exp[ix.A] f(\cdot), \quad x\in\Z^d, \ f\in L^2(\Om),
\ee
where $A=(A_1,..,A_d)$. It follows then from (\ref{G2}) and (\ref{AB2}) that
\be \label{AC2}
T_{\xi,\eta} g(\cdot) \ = \  \frac{e(\xi-A)e^*(\xi-A)}{\eta/\La+e(\xi-A)^*e(\xi-A)}  \ g(\cdot) \ .
\ee

It is easy to see that the function $q(\xi,\eta)$ is $C^\infty$ for $\xi\in\R^d, \ \eta>0$.  In \cite{cn1,cn2} it was further shown that if the translation operators $\tau_x, \ x\in\Z^d$, are ergodic on  $(\Om,\mathcal{F},P)$ then $\lim_{(\xi,\eta)\ra (0,0)}q(\xi,\eta)=q(0,0)$ exists. We can extend this result as follows:
\begin{proposition}
Suppose that the operator  $\tau_{{\bf e}_j}$ is weak mixing on $\Om$ for some $j, \ 1\le j\le d$. Then  $q(\xi,\eta),  \  \xi\in\R^d,\ \eta>0$, extends to a continuous function on $\xi\in\R^d, \ \eta\ge 0$. 
\end{proposition}
\begin{proof}
We first define an operator $T_{\xi,\eta} $ on $\mathcal{H}(\Om)$ for $\xi\in\R^d$ and $\eta=0$. To do this first observe from (\ref{F2}) that if the function $g(\cdot)\in \mathcal{H}(\Om)$ satisfies $\pa^*_\xi g(\cdot)=0$, we should set  $T_{\xi,0}g(\cdot)=0$. Alternatively if $g(\cdot)=\pa_\xi h(\cdot)$ for some $h\in L^2(\Om)$, then from (\ref{G2}) we should set  $T_{\xi,0}g(\cdot)$ to be given by the formula,
\be \label{J2}
T_{\xi,0} g(\cdot) \ = \   \sum_{x\in \Z^d} \lim_{\eta\ra 0}\left\{\nabla^*(\nabla^*\nabla) G_{\eta/\La}(x)\right\}^*\exp[-ix.\xi]  \ h(\tau_x\cdot) \ .
\ee 
In view of the inequality
\be \label{K2}
\left| \nabla^*(\nabla^*\nabla) G_{\eta}(x)\right| \ \le \ C/[1+|x|]^{d+1} \quad x\in\Z^d, \ \eta>0,
\ee
for a constant $C$ depending only on $d$, we see that the right hand side of (\ref{J2}) is in $\mathcal{H}(\Om)$. Since the orthogonal complement in $\mathcal{H}(\Om)$ of the null space of the operator $\pa^*_\xi$ is the closure of the linear space $\mathcal{E}_\xi(\Om)=\{\pa_\xi h(\cdot):  h\in L^2(\Om)\}$, we have defined  $T_{\xi,0} g(\cdot) $ for a dense set of functions  $g\in\mathcal{H}(\Om)$ and 
\be \label{L2}
\lim_{\eta\ra 0} \|T_{\xi,\eta} g-T_{\xi,0} g\| \ = \ 0 \ .
\ee
Using the fact that $\|T_{\xi,\eta}\|\le 1$ for all $\eta>0$, one sees that the operator $T_{\xi,0}$, defined above on a dense linear subspace of $\mathcal{H}(\Om)$, extends to a bounded operator on $\mathcal{H}(\Om)$ with norm  $\|T_{\xi,0}\|\le 1$, and the limit (\ref{L2}) holds for all $g\in \mathcal{H}(\Om)$.

Having extended the operators $T_{\xi,\eta}, \ \xi\in\R^d,\eta>0$, to the region $\xi\in\R^d, \ \eta\ge 0$, we can use this fact to similarly extend the function $q(\xi,\eta)$. Thus   for $m=1,2...$, let the matrix function $h_m(\xi,\eta)$ be defined for $\eta> 0, \ \xi\in \R^d$, by
\be \label{M2}
h_m(\xi,\eta)=\langle \ {\bf b}(\cdot)\left[ PT_{\xi,\eta}{\bf b}(\cdot)\right]^m \ \rangle \ ,
\ee
whence (\ref{E2}), (\ref{I2}) imply that
\be \label{N2}
q(\xi,\eta) \ = \  \langle \ {\bf a}(\cdot) \ \rangle-\La \sum_{m=1}^\infty h_m(\xi,\eta) \ .
\ee
Evidently we can extend the definition of  $q(\xi,\eta)$ to $\eta=0$ by setting $\eta=0$ in (\ref{M2}), (\ref{N2}). In view of (\ref{L2}) one has that $\lim_{\eta\ra 0} q(\xi,\eta)=q(\xi,0), \ \xi\in\R^d$, whence  the inequality (\ref{B2}) continues to hold in the extended region. 

Assume now that the operator $\tau_{{\bf e}_j}$ is weak mixing on $\Om$ for some $j, \ 1\le j\le d$, and let  $\mathcal{E}_{j,\xi}(\Om)=\{ \ \pa_{j,\xi} g(\cdot):  g\in \mathcal{H}(\Om)   \}$. Then \cite{p} the closure of $\mathcal{E}_{j,\xi}(\Om)\subset \mathcal{H}(\Om)$ contains the orthogonal complement of the constants i.e. $\{g\in \mathcal{H}(\Om): \langle \ g(\cdot) \ \rangle=0\}\subset \bar{\mathcal{E}}_{j,\xi}(\Om)$. Suppose now $g\in\mathcal{H}(\Om)$ and $\langle \ g(\cdot) \ \rangle=0$. Then for any $\ve>0$ there exists $\del>0$ depending only on $\ve,\xi$ and $g(\cdot)$, but not on $\eta>0$, such that
 \be \label{O2}
 \|T_{\xi',\eta} g-T_{\xi,\eta} g\| \ < \ \ve,
 \ee
 provided $|\xi'-\xi|<\del$. To see this first note  that there exists  $g_{\ve,\xi}\in \mathcal{E}_{j,\xi}(\Om)$ which satisfies  $\|g-\pa_{j,\xi} g_{\ve,\xi}\|<\ve/3$. Next  observe that
 \be \label{P2}
 \| \ T_{\xi',\eta}\pa_{j,\xi} g_{\ve,\xi}\ -T_{\xi,\eta} \pa_{j,\xi} g_{\ve,\xi} \ \| \ \le  \ 
  \| \ T_{\xi',\eta}\pa_{j,\xi'} g_{\ve,\xi}\ -T_{\xi,\eta} \pa_{j,\xi} g_{\ve,\xi} \ \| + C|\xi'-\xi| \ \|g_{\ve,\xi}\|,
 \ee
where the constant $C$ depends only on $d$. It also follows from (\ref{K2}) that there is a constant $C$ depending only on $d$ such that
\be \label{Q2}
 \| \ T_{\xi',\eta}\pa_{j,\xi'} g_{\ve,\xi}\ -T_{\xi,\eta} \pa_{j,\xi} g_{\ve,\xi} \ \|  \ \le \ C|\xi'-\xi|^{1/2} \ \|g_{\ve,\xi}\| \ .
\ee
The inequality (\ref{O2}) follows from (\ref{P2}), (\ref{Q2}) on choosing $\del$ sufficiently small independent of $\eta>0$. 
 
 Since $\langle \ P{\bf b}(\cdot) \ \rangle=0$, the continuity of the function $h_1(\xi,\eta)$ in the region $\xi\in\R^d, \ \eta\ge 0$, immediately follows from (\ref{O2}). The continuity of the functions $h_m(\cdot,\cdot), \ m\ge 1$, follow similarly on using the uniform bound  $\|T_{\xi,\eta}\|\le 1, \ \xi\in\R^d,\eta\ge 0$.
\end{proof}
\begin{remark}
Note that the projection operator $P$ in the formula (\ref{M2}) plays a critical role in establishing  continuity. For a constant function $g(\cdot)\equiv v\in\C^d$, one has
\be \label{R2}
T_{\xi,\eta} g(\cdot) \ = \  [e(\xi)^*v]e(\xi) \big/ [\eta/\La+e(\xi)^*e(\xi)] \ ,
\ee
which does not extend to a continuous function of $(\xi,\eta)$ on the set $\xi\in\R^d, \ \eta\ge 0$.
\end{remark}
Next we show that the function $q(\xi,\eta)$ with domain $\xi\in\R^d, \  \eta>0$, can be extended to complex $\xi=\Re\xi+i\Im\xi\in\C^d$ with small imaginary part.
\begin{lem}
For fixed $\eta$ satisfying $0<\eta\le \La$, the $C^\infty$ operator valued function $\xi\ra T_{\xi,\eta}$ from $\R^d$ to the space of bounded linear operators $\mathcal{B}[\mathcal{H}(\Om)]$ on  $\mathcal{H}(\Om)$ has an analytic continuation to a region $\{\xi\in\C^d: |\Im\xi|<C_1\sqrt{\eta/\La}\}$, where $C_1$ is a constant depending only on $d$. For $\xi$ in this region the norm of $T_{\xi,\eta}$ satisfies the inequality $\|T_{\xi,\eta}\|\le 1+C_2 |\Im\xi|^2/[\eta/\La]$, where the constant $C_2$ depends only on $d$.  
\end{lem} 
\begin{proof}
The fact that there is an analytic continuation to the region $\{\xi\in\C^d: |\Im\xi|<C_1\sqrt{\eta/\La}\}$ is a consequence of the bound on the function $G_\eta(\cdot)$ of (\ref{H2}), 
\be \label{S2}
\left|\nabla\nabla^* G_{\eta}(x)\right| \ \le \ \frac{C_3\exp[-C_4\sqrt{\eta} \ |x|]}{[1+|x|]^d} \ , \quad x\in\Z^d, \ 0<\eta\le 1,
\ee
where the constants $C_3,C_4$ depend only on $d$. The bound on $\|T_{\xi,\eta}\|$ can be obtained from (\ref{F2}).  Thus on multiplying (\ref{F2}) by $\bar{\psi}(\xi,\eta,\om)$, taking the expectation and using the Schwarz inequality, we see that
\begin{multline} \label{T2}
\left\{ \ 1-C_5|\Im\xi|^2/[\eta/\La]  \ \right\}\|\pa_{\Re \xi}\psi(\xi,\eta,\cdot)\|^2 +\frac{\eta}{2\La} \|\psi(\xi,\eta,\cdot)\|^2 \  \le \\
\left\{ \ \frac{1}{2}+C_6|\Im\xi|^2/[\eta/\La]  \ \right\} \|g\|^2 + \frac{1}{2}\|\pa_{\Re \xi}\psi(\xi,\eta,\cdot)\|^2 +\frac{\eta}{4\La} \|\psi(\xi,\eta,\cdot)\|^2 \ , 
\end{multline}
where the constants $C_5,C_6$ depend only on $d$.  Evidently (\ref{T2}) yields the bound on $\|T_{\xi,\eta}\|$ on taking $C_1$ sufficiently small, depending only on $d$.
\end{proof}
\begin{corollary}
For fixed $\eta$  satisfying $0<\eta\le \La$, the $d\times d$ matrix function $q(\xi,\eta)$ with domain $\xi\in\R^d$, has an analytic continuation to a region $\{\xi\in\C^d: |\Im\xi|<C_1\sqrt{\la\eta/\La^2}\}$, where $C_1$ is a constant depending only on $d$. There is a constant $C_2$ depending only on $d$ such that for $\xi$ in this region, 
\be \label{U2}
\|q(\xi,\eta)-q(\Re\xi,\eta)\| \ \le \ \frac{C_2\La^2}{\la} \frac{ |\Im\xi|}{\sqrt{\eta/\La}} \ .
\ee
\end{corollary}
\begin{proof} The fact that $q(\xi,\eta)$ has an analytic continuation follows from the representations (\ref{M2}), (\ref{N2}), Lemma 2.1 and the matrix norm bound $\|{\bf b}(\om)\|\le 1-\la/\La, \ \om\in\Om$. On summing the perturbation series (\ref{N2}), we conclude that for $\xi$ satisfying  $ |\Im\xi|<C_1\sqrt{\la\eta/\La^2}$, then $\|q(\xi,\eta)\|\le  C_2\La^2/\la$ for a constant $C_2$ depending only on $d$, provided $C_1$ is chosen sufficiently small, depending only on $d$. By arguing as in Lemma 2.1 we also see that there are positive constants $C_1, C_2$ such  that 
\be \label{V2}
\|T_{\xi,\eta}-T_{\Re\xi,\eta}\| \ \le \  C_2 |\Im\xi|/\sqrt{\eta/\La} \ , \quad \xi\in\C^d, \  |\Im\xi|<C_1\sqrt{\eta/\La} \ .
\ee
The inequality (\ref{U2}) follows from (\ref{V2}).
\end{proof}
We have seen in Corollary 2.1 that the periodic matrix function $q(\xi+ia,\eta), \ \xi\in [-\pi,\pi]^d$, is bounded provided $a\in\R^d$ satisfies $|a|\le C_1\sqrt{\la\eta/\La^2}$. It was shown in \cite{cn2} that derivatives of this function are in certain weak $L^p$ spaces. For $1\le p<\infty$ the weak $L^p$ space $L^p_w([-\pi,\pi]^d)$ is defined to be all measurable functions $f:[-\pi,\pi]^d\ra\C$ such that  
\be \label{W2}
{\rm meas}\{\xi\in[-\pi,\pi]^d:|f(\xi)|>\mu\} \ \le \ C^p/\mu^p \ , \quad \mu>0.
\ee
The weak $L^p$ norm of $f(\cdot)$, $\|f\|_{p,w}$ is the minimum constant $C$ such that (\ref{W2}) holds. 
From \cite{cn2} we have the following:
\begin{proposition}
Let $d\ge 1,0<\eta\le \La, \ 1\le k,k'\le d$, and $m=(m_1,..,m_d)$ be a $d-$tuple of non-negative integers with norm $|m|=m_1+\cdots+m_d$. Then there exists a positive constant $C_1$ depending only on $d$ such that if $|a|\le C_1\sqrt{\la\eta/\La^2}$ and $|m|<d$, the function 
\be \label{X2}
\prod_{j=1}^d \left(\frac{\pa}{\pa \xi_j}\right)^{m_j} q_{k,k'}(\xi+ia,\eta), \quad  \xi\in[-\pi,\pi]^d \ ,
\ee
is in the space $L^p_w([-\pi,\pi]^d)$ with $p=d/|m|$ and its norm is bounded by $C\La $, where the constant $C$ depends only on $d$ and $\La/\la\ge 1$.  

If $|m|=d-1$ and $0<\del\le 1$ then for any $\rho\in\R^d$ satisfying $|\rho|\le 1$, the function 
\be \label{Y2}
\prod_{j=1}^d \left(\frac{\pa}{\pa \xi_j}\right)^{m_j} \left[ \ q_{k,k'}(\xi+\rho+ia,\eta)-q_{k,k'}(\xi+ia,\eta) \ \right]/|\rho|^{1-\del}, \quad  \xi\in[-\pi,\pi]^d \ ,
\ee
is in the space $L^p_w([-\pi,\pi]^d)$ with $p=d/(d-\del)$ and its norm is bounded by $ C_\del\La$, where the constant $C_\del$ depends only on $d, \ \La/\la$ and $\del>0$. 
\end{proposition}
\section{Configuration space estimates from Fourier space estimates}
In this section we shall show how to obtain configuration space estimates on $G_{{\bf a},\eta}(x), \ x\in\Z^d$, from  Fourier space estimates on the function $q(\xi,\eta), \ \xi\in\R^d$. In \cite{cn2} it was shown that for $d\ge 3$, Proposition 2.2 implies the inequality
\be \label{A3}
0 \ \le G_{{\bf a},\eta}(x) \ \le  \   \frac{C\exp[-\ga |x| \sqrt{\eta/\La}]}{\La( |x|+1)^{d-2}}, \quad x\in\Z^d, \ 0<\eta\le \La,
\ee 
where the positive constants $C, \ \ga$ depend only on $d$ and $\La/\la$. It was also shown that for $d\ge 2$, Proposition 2.2 implies a similar inequality  for the gradient of $G_{{\bf a},\eta}(x)$,
\be \label{B3}
|\na G_{{\bf a},\eta}(x)| \ \le  \   \frac{C\exp[-\ga |x| \sqrt{\eta/\La}]}{\La( |x|+1)^{d-1}}, \quad x\in\Z^d, \ 0<\eta\le \La.
\ee 
Finally for $d\ge 1$, Proposition 2.2 implies H\"{o}lder continuity of  $\na G_{{\bf a},\eta}(x)$,
\begin{multline} \label{C3}
|\na G_{{\bf a},\eta}(x')-\na G_{{\bf a},\eta}(x)| \ \le  \   |x'-x|^{1-\del}\frac{C_\del\exp[-\ga |x| \sqrt{\eta/\La}]}{\La( |x|+1)^{d-\del}}, \\
  0<\eta\le \La, \quad  \ x',x\in\Z^d, \  1/2\le (|x'|+1)/(|x|+1)\le 2, 
\end{multline} 
for any $\del, \ 0<\del\le 1$, where the constant $C_\del$ depends now on $\del>0$ as well as on $d$ and $\La/\la$. 

In subsequent sections we shall establish  a strengthened version of Proposition 2.2 for certain environments $(\Om,\mathcal{F},P)$ as follows:
\begin{theorem}
There exist positive constants $C_1,C_2$ and  $\al\le1$ depending only on $d$ and $\La/\la$,  such that
\begin{multline} \label{D3}
\|q(\xi',\eta')-q(\xi,\eta)\|\le C_1\La \left[ \ |\xi'-\xi|^\alpha +|(\eta'-\eta)/\La|^{\al/2} \ \right]  \ , \\
0<\eta\le \eta'\le \La, \quad \xi',\xi\in\C^d \ {\rm with \ } |\Im\xi|+|\Im\xi'|\le C_2\sqrt{\eta/\La} \ .
\end{multline}
With the same assumptions as in Proposition 2.2,  the derivative (\ref{X2}) is in the space $L^p_w([-\pi,\pi]^d)$ with $p=d/(|m|-\al)$ and its norm is bounded by $C\La$, where the constant $C$ depends only on $d$ and $\La/\la$.  The difference (\ref{Y2}) is in the space $L^p_w([-\pi,\pi]^d)$ with $p=d/(d-\del-\al)$ and its norm is bounded by $C\La$, where now $\al$ and $C$ depend on $\del$ as well as $d$ and $\La/\la$. 
\end{theorem}
Theorem 3.1 enables us to compare the function $G_{{\bf a},\eta}(x),  \  x\in\Z^d$, to the function $G_{{\bf a}_{\rm hom},\eta}(x),  \  x\in\Z^d$, defined by  
\be \label{E3}
G_{{\bf a}_{\rm hom},\eta}(x) \ = \ 
\frac{1}{(2\pi)^d}\int_{[-\pi,\pi]^d} 
\frac{e^{-i\xi.x}}{\eta+e(\xi)^*q(0,0)e(\xi)} \ d\xi \ .
\ee
Note that  the function $G_{{\bf a}_{\rm hom},\eta}(\cdot)$  of (\ref{E3}) is not the same as the Green's function for the PDE (\ref{D1}) restricted to $\Z^d$. One can however easily estimate the difference on $\Z^d$ between these two functions. 
\begin{theorem}
For $\eta$ satisfying $0<\eta\le\La,$ there exist positive constants $\al,\ga,C$ with  $\al\le 1$, depending only on $d$ and $\La/\la$ such that
\begin{eqnarray} \label{F3}
|G_{{\bf a},\eta}(x)-G_{{\bf a}_{\rm hom},\eta}(x)|  &\le&  \frac{C}{\La(|x|+1)^{d-2+\alpha}} e^{-\gamma\sqrt{\eta/\La} |x|}, \    \ x\in\Z^d,  \\
 |\na G_{{\bf a},\eta}(x)-\na G_{{\bf a}_{\rm hom},\eta}(x)|  &\le& \frac{C}{\La(|x|+1)^{d-1+\alpha}} e^{-\gamma\sqrt{\eta/\La} |x|}, \   \ x\in\Z^d,  \label{G3}
\end{eqnarray}
where (\ref{F3}) holds if $d\ge 2$ and (\ref{G3}) if $d\ge 1$.
For $d\ge 1$ and $\del$ satisfying $0<\del\le 1$, there is the inequality
\begin{multline} \label{H3}
\big| \ [\na G_{{\bf a},\eta}(x')-\na G_{{\bf a}_{\rm hom},\eta}(x')]-[\na G_{{\bf a},\eta}(x)-\na G_{{\bf a}_{\rm hom},\eta}(x)] \ \big|  \\
\le \ |x'-x|^{1-\del} \frac{C}{\La(|x|+1)^{d-\del+\alpha}} e^{-\gamma\sqrt{\eta/\La} |x|}, \quad   \ x',x\in\Z^d, \ 
1/2\le (|x'|+1)/(|x|+1)\le 2, 
\end{multline}
where $C$ and $\al$ in (\ref{H3}) depend on $\del>0$ as well as $d$ and $\La/\la$. In particular $\al<\del$. 
\end{theorem} 
\begin{proof} Let $\tilde{G}_{{\bf a}_{\rm hom},\eta}(x)$ be the function
\be \label{J3}
\tilde{G}_{{\bf a}_{\rm hom},\eta}(x) \ = \ 
\frac{1}{(2\pi)^d}\int_{[-\pi,\pi]^d} 
\frac{e^{-i\xi.x}}{\eta+e(\xi)^*q(0,\eta)e(\xi)} \ d\xi \ ,
\ee
so (\ref{A2}) and Corollary 2.1 imply that
\be \label{K3}
G_{{\bf a},\eta}(x)-\tilde{G}_{{\bf a}_{\rm hom},\eta}(x) \ = \ \frac{e^{a.x}}{(2\pi)^d}\int_{[-\pi,\pi]^d} 
e^{-i\xi.x} f(\xi) \ d\xi \ , 
\ee
where the function $f(\xi)$ is given by the formula
\be \label{L3}
f(\xi) \ =  \ \frac{e(\xi-ia)^*\{q(0,\eta)-q(\xi+ia,\eta)\}e(\xi+ia)   }{\left[\eta+e(\xi-ia)^*q(0,\eta)e(\xi+ia)\right] 
\left[\eta+e(\xi-ia)^*q(\xi+ia,\eta)e(\xi+ia)\right]}  \ .
\ee
It follows from (\ref{D3}) and Corollary 2.1  that there are positive constants $C_1,C_2$ depending only on $d$ and $\La/\la$ such that  the function in (\ref{L3}) is bounded by 
\be \label{M3}
|f(\xi)| \ \le \ \frac{C_1}{\La[\eta/\La+|e(\xi)|^2]^{1-\al/2}} \ , \quad {\rm for \ } |a|\le C_2\sqrt{\eta/\La} \ .
\ee
Choosing $a$ appropriately and $\al<1$ one sees already from (\ref{K3}), (\ref{M3}) that if $d=1$ the function $|G_{{\bf a},\eta}(x)-\tilde{G}_{{\bf a}_{\rm hom},\eta}(x)|$ is bounded by the RHS of (\ref{F3}) provided $\sqrt{\eta/\La}|x|\ge 1$. 

We wish to show for $d=1$ that $|\na G_{{\bf a},\eta}(x)-\na\tilde{G}_{{\bf a}_{\rm hom},\eta}(x)|$ is bounded by the RHS of (\ref{G3}). The representation of $\na G_{{\bf a},\eta}(x)-\na\tilde{G}_{{\bf a}_{\rm hom},\eta}(x)$ corresponding to (\ref{K3}) is 
\be \label{N3}
\na G_{{\bf a},\eta}(x)-\na\tilde{G}_{{\bf a}_{\rm hom},\eta}(x) \ = \ \frac{e^{a.x}}{(2\pi)^d}\int_{[-\pi,\pi]^d} 
e^{-i\xi.x} e(\xi)f(\xi) \ d\xi \ .
\ee
 Let $\rho\in\R^d$ be the vector of minimum norm which satisfies $e^{-i\rho.x}=-1$. Then we have that
\be \label{O3}
\big| \ \int_{[-\pi,\pi]^d} e^{-i\xi.x} e(\xi) f(\xi) \ d\xi \ \big|  \le \ \frac{1}{2}  \int_{[-\pi,\pi]^d} |e(\xi) f(\xi)-e(\xi+\rho)f(\xi+\rho)| \ d\xi  \ .
\ee
The integral on the RHS of (\ref{O3}) can be estimated by separately estimating the integral  over the region $|\xi|<C/[|x|+1]$ and $|\xi|>C/[|x|+1]$ for sufficiently large universal constant $C$. Using the estimate for $|f(\xi)|$ obtained from (\ref{M3}) and the fact that $|\rho|\le C'/[|x|+1]$, we see that the integral over $|\xi|<C/[|x|+1]$  is bounded by a constant times $[|x|+1]^{-\al}$. To estimate the integral over $|\xi|>C/[|x|+1]$ we use the H\"{o}lder continuity of the function $f(\cdot)$. From (\ref{D3}) it follows that
\be \label{P3}
|e(\xi) f(\xi)-e(\xi+\rho)f(\xi+\rho)|  \  \le \  \frac{C_1|\rho|^\al}{\La[\eta/\La+|e(\xi)|^2]^{1/2}}  \ ,\quad |\xi|>C/[|x|+1] \ ,
\ee
for a constant $C_1$ depending only on $d$ and $\La/\la$, provided $C$ is sufficiently large. Hence 
 the integral over $|\xi|>C/[|x|+1]$  is bounded by a constant times $[|x|+1]^{-\al}\log[|x|+2]$. We have shown that $|\na G_{{\bf a},\eta}(x)-\na\tilde{G}_{{\bf a}_{\rm hom},\eta}(x)|$ is bounded by the RHS of (\ref{G3}) for any $\alpha$ less than the H\"{o}lder constant in (\ref{D3}). We can similarly bound  $|\na  \tilde{G}_{{\bf a}_{\rm hom},\eta}(x)-\na G_{{\bf a}_{\rm hom},\eta}(x)|$ by the RHS of (\ref{G3}) on using the estimate $\|q(0,\eta)-q(0,0)\|\le C\La(\eta/\La)^{\al/2}$ from (\ref{D3}). Hence (\ref{G3}) holds for $d=1$, and by similar argument  (\ref{H3}).  

In order to prove (\ref{F3}) for $d\ge 2$ we need to use the bounds on the derivatives of the function $q(\cdot,\cdot)$ given in Proposition 2.2. For $d=2$ the first derivative estimate is sufficient. Thus we write
\begin{multline} \label{Q3}
\int_{[-\pi,\pi]^d} e^{-i\xi.x} f(\xi) \ d\xi \ = \ \int_{|\xi|<C/[|x|+1]} e^{-i\xi.x} f(\xi) \ d\xi  + \\
\frac{i}{|x|^2}\int_{|\xi|=C/[|x|+1]} [x.{\bf n}(\xi)] \ e^{-i\xi.x} f(\xi) \ d\xi -
\frac{i}{|x|^2}\int_{|\xi|>C/[|x|+1]}  \ e^{-i\xi.x} \  [x.\na_\xi f(\xi)] \ d\xi \ , 
\end{multline}
where ${\bf n}(\xi)$ is the unit inward normal vector  at $\xi$ on the sphere $\{|\xi|=C/[|x|+1]\}$.  It follows from (\ref{M3}) that the first two integrals on the RHS of (\ref{Q3}) are bounded by $C_2/\La[|x|+1]^{d-2+\al}$ for some constant $C_2$ depending only on $d$ and  $\La/\la$. The third integral can be similarly bounded for $d=2$ by using the fact from Theorem 3.1 that the function $\na_\xi q(\xi+ia,\eta)$ is in $L_w^p([-\pi,\pi]^2)$ with $p=2/(1-\al)$. To see this we note that for any  measure space $(X,\mathcal{B},\mu)$  and $1< p<\infty$ one has that if  $g\in L_w^p(X)$ then
\be \label{R3}
\int_E |g| d\mu \ \le \ C_p\|g\|_{p,w} \  m(E)^{1-1/p} \ , \qquad E\in\mathcal{B} ,
\ee
where $C_p$ depends only on $p$. It follows that for $n=1,2,..$,
\be \label{S3}
\int_{2^{n-1}C/[|x|+1]<|\xi|<2^nC/[|x|+1]} |\na_\xi f(\xi)| \ d\xi   \le \  \frac{C_1 2^{-(1-\al)n}}{\La[|x|+1]^{\al-1}} \ ,
\ee
where $C_1$ depends only on $d=2$ and $\La/\la$. Hence on summing over $n\ge 1$ in (\ref{S3}) we see that for $d=2$ the function $|G_{{\bf a},\eta}(x)-\tilde{G}_{{\bf a}_{\rm hom},\eta}(x)|$ is bounded by the RHS of (\ref{F3}). We can bound  $| \tilde{G}_{{\bf a}_{\rm hom},\eta}(x)-G_{{\bf a}_{\rm hom},\eta}(x)|$ using the H\"{o}lder continuity (\ref{D3}) of $q(0,\eta), \ 0\le \eta\le \La$,  by the RHS of (\ref{F3}) for any $\al$ smaller than the H\"{o}lder constant in (\ref{D3}).  We have therefore proven (\ref{F3}) for $d=2$. 

To prove (\ref{G3}) for $d=2$ we need to use the fact from Theorem 3.1 that the difference $[\na_\xi q(\xi+\rho+ia,\eta)-\na_\xi q(\xi+ia,\eta)]/|\rho|^{1-\del}$ is in  $L_w^p([-\pi,\pi]^2)$ with $p=2/(2-\del-\al)$ and norm bounded independent of $|\rho|\le 1$. Thus in bounding $\na G_{{\bf a},\eta}(x)-\na\tilde{G}_{{\bf a}_{\rm hom},\eta}(x)$ we write the integral over $\xi$ as in (\ref{Q3}). The first two terms on the RHS  can be bounded in a straightforward way. To bound the third term we need to estimate the integral
\begin{multline} \label{T3}
\int_{|\xi|>C/[|x|+1]}  \ e^{-i\xi.x} \  (x.\na_\xi) [e(\xi)f(\xi)] \ d\xi \ = \\
\frac{1}{2}\int_{|\xi|>C'/[|x|+1]}  \ e^{-i\xi.x} \  (x.\na_\xi) [e(\xi)f(\xi)-e(\xi+\rho)f(\xi+\rho)] \ d\xi  \ + \ {\rm Error}.
\end{multline}
In (\ref{T3}) the vector $\rho\in\R^d$ is as in (\ref{O3}) and $C'$ is a universal constant. The error term in (\ref{T3})  is bounded by
\be \label{U3}
{\rm Error} \ \le \ \int_{C/10[|x|+1] <|\xi|<10C/[|x|+1]}  \big| (x.\na_\xi) [e(\xi)f(\xi)]\big| \ d\xi \ ,
\ee
which can be appropriately estimated by using the fact that $\na_\xi q(\xi+ia,\eta)$ is in $L_w^p([-\pi,\pi]^2)$ with $p=2/(1-\al)$. 
Similarly to (\ref{S3}) we have for $n=1,2,..$, the bound
\be \label{V3}
\int_{2^{n-1}C'/[|x|+1]<|\xi|<2^nC'/[|x|+1]}\big|\na_\xi [e(\xi)f(\xi)-e(\xi+\rho)f(\xi+\rho)] \big|  \ d\xi   \le \  \frac{C_\del 2^{-(1-\al-\del)n}|\rho|^{1-\del}}{\La[|x|+1]^{\al+\del-1}} \ ,
\ee
where $C_\del$ depends only on $d=2$ and $\La/\la$.  Choosing now $\al<1-\del$ in (\ref{V3}),  we conclude that  $|\na G_{{\bf a},\eta}(x)-\na\tilde{G}_{{\bf a}_{\rm hom},\eta}(x)|$ is bounded by the RHS of (\ref{G3}). As in the previous paragraph,  we can bound  $| \na \tilde{G}_{{\bf a}_{\rm hom},\eta}(x)-\na G_{{\bf a}_{\rm hom},\eta}(x)|$ using the H\"{o}lder continuity (\ref{D3}) of $q(0,\eta), \ 0\le \eta\le \La$,  by the RHS of (\ref{G3}) for any $\al$ smaller than the H\"{o}lder constant in (\ref{D3}). We have therefore proven (\ref{G3}) for $d=2$.

We proceed similarly for the proof of (\ref{H3}) in the case $d=2$. Thus we write
\begin{multline} \label{W3}
\int_{|\xi|>C/[|x|+1]}  \ [e^{-i\xi.x}-e^{-i\xi.x'}] \  (x.\na_\xi) [e(\xi)f(\xi)] \ d\xi \ = \\
\frac{1}{2}\int_{|\xi|>C'/[|x|+1]}  \ [e^{-i\xi.x}-e^{-i\xi.x'}] \  (x.\na_\xi) [e(\xi)f(\xi)-e(\xi+\rho)f(\xi+\rho)] \ d\xi  \ +\\
\frac{1}{2}[1-e^{i\rho.(x'-x)}]\int_{|\xi|>C'/[|x|+1]}  \ e^{-i\xi.x} \  (x.\na_\xi) [e(\xi+\rho)f(\xi+\rho)] \ d\xi  \ +
{\rm Error}.
\end{multline}
where $\rho\in\R^d$ is the vector of minimum norm such that $e^{-i\rho.x'}=-1$.
Arguing as in (\ref{U3}) we see that
\be \label{X3}
| \ {\rm Error} \ | \ \le \  C_1|x-x'|/\La[|x|+1]^\al,
\ee
where $C_1$ depends only on $\La/\la$. From (\ref{V3}) we see that the first term on the RHS of (\ref{W3}) is bounded by the sum 
\be \label{Y3}
|x||x-x'|^{1-\del}\sum_{n=1}^\infty \left(\frac{2^nC'}{[|x|+1]}\right)^{1-\del}\frac{C_{\del'} 2^{-(1-\al-\del')n}|\rho|^{1-\del'}}{\La[|x|+1]^{\al+\del'-1}} \ ,
\ee
for any $\del'$ satisfying $0<\del'\le 1$. Provided $\al<\del$ we may choose $\del'>0$ so that the sum in (\ref{Y3}) converges, whence the sum is bounded by $C_\del|x-x'|^{1-\del}/\La[|x|+1]^{\al-\del}$ for a constant $C_\del$ depending only on $\del>0$ as well as $\La/\la$. The second term on the RHS of (\ref{W3}) is $[1-e^{i\rho.(x'-x)}]$ times an integral similar to the integral on the LHS of (\ref{T3}). Hence the second term is bounded in absolute value by $C_1|x-x'|/\La[|x|+1]^\al$, where $C_1$ depends only on $\La/\la$.  We have therefore shown that  $|\na G_{{\bf a},\eta}(x)-\na\tilde{G}_{{\bf a}_{\rm hom},\eta}(x)|$ is bounded by the RHS of (\ref{H3}). Since we can argue as previously to bound $| \na \tilde{G}_{{\bf a}_{\rm hom},\eta}(x)-\na G_{{\bf a}_{\rm hom},\eta}(x)|$, the proof of (\ref{H3}) is complete. 

   To prove the result for $d\ge 3$ we use multiple integration by parts and the integrability properties of the higher derivatives of $q(\xi+ia,\eta), \ \xi\in[-\pi,\pi]^d$, given in Theorem 3.1.
\end{proof}

\vspace{.2in}
\section{Independent Variable Environment}
Our goal in this section will be to prove Theorem 3.1 in the case when the variables ${\bf a}(\tau_x\cdot), \ x\in\Z^d$, are independent.  Following \cite{cn1} we first consider the case of a Bernoulli environment.  Thus for each $n\in\Z^d$ let $Y_n$ be independent Bernoulli variables, whence $Y_n=\pm 1$ with equal probability. The probability space $(\Om,\mathcal{F},P)$ is then the space generated by all the variables $Y_n, \ n\in\Z^d$.  A point $\om\in\Om$ is a set of configurations $\{(Y_n,n): n\in\Z^d\}$. For $y\in\Z^d$ the translation operator $\tau_y$ acts on $\Om$ by taking the point $\om=\{(Y_n,n): n\in\Z^d\}$ to $\tau_y\om=\{(Y_{n+y},n): n\in\Z^d\}$. The random matrix ${\bf a}(\cdot)$ is then defined by 
\be \label{A4}
{\bf a}(\om) \ = \ (1+\ga Y_0) I_d, \quad \om=\{(Y_n,n): n\in\Z^d\} \ ,
\ee
where $0\le \ga<1$.
In \cite{cn1} we defined for $1\le p<\infty$ Fock spaces $\mathcal{F}^p(\Z^d)$  of complex valued functions, and observed that $\mathcal{F}^2(\Z^d)$ is unitarily equivalent to $L^2(\Om)$. We can similarly define Fock spaces $\mathcal{H}_\mathcal{F}^p(\Z^d)$ of vector valued functions with domain $\C^d$, such that $\mathcal{H}_\mathcal{F}^2(\Z^d)$  is unitarily equivalent to  $\mathcal{H}(\Om)$. Hence we can regard the operator $T_{\xi,\eta}$ of (\ref{G2}) as acting on $\mathcal{H}_\mathcal{F}^2(\Z^d)$, and by unitary equivalence it is a bounded operator satisfying $\|T_{\xi,\eta}\|\le 1$ for $\xi\in\R^d, \ \eta>0$.  We may apply now the Calderon-Zygmund theorem \cite{stein} to conclude the following:
\begin{lem}
For $\xi\in\R^d, \ 0<\eta\le 1$, and $1<p<\infty$, the operator $T_{\xi,\eta}$ is a bounded operator on $\mathcal{H}_\mathcal{F}^p(\Z^d)$ with norm $\|T_{\xi,\eta}\|_p$ satisfying an inequality $\|T_{\xi,\eta}\|_p\le 1+\del(p) $, where $\lim_{p\ra 2}\del(p)=0$.  
\end{lem} 
It is well known  for the independent variable environment $(\Om,\mathcal{F},P)$ that the operators $\tau_{{\bf e}_j}, \ j=1,..d$, are strong mixing on $\Om$. Hence Proposition 2.1 implies that the function  $q(\xi,\eta)$ with domain $ \xi\in[-\pi,\pi]^d, \ 0<\eta\le\La$, is uniformly continuous. Lemma 4.1 enables us improve this result to uniform H\"{o}lder continuity.   
\begin{proposition}
The function $q(\xi,\eta)$  of (\ref{E2}) with domain $\xi\in[-\pi,\pi]^d, \ 0<\eta\le\La$,  is uniformly H\"{o}lder continuous. That is there exist positive constants $C,\alpha$  with $0<\alpha\le 1$ depending only on $d$ and $\La/\la$, such that $\|q(\xi',\eta')-q(\xi,\eta)\|\le C\La \left[ \ |\xi'-\xi|^\alpha +|(\eta'-\eta)/\La|^{\al/2} \ \right]$ for  $\xi',\xi\in[-\pi,\pi]^d$ and $0<\eta,\eta' \le \La$.
\end{proposition} 
\begin{proof}
We use the representation (\ref{M2}), (\ref{N2}) for $q(\xi,\eta)$. From (\ref{M2}) we have that
\begin{multline} \label{B4}
h_m(\xi',\eta)-h_m(\xi,\eta) \ = \\
\sum_{j=0}^{m-1}\langle \ {\bf b}(\cdot)\left[ PT_{\xi',\eta}{\bf b}(\cdot)\right]^{m-1-j} \  P
\left[ T_{\xi',\eta}- T_{\xi,\eta}\right]{\bf b}(\cdot)\left[ PT_{\xi,\eta}{\bf b}(\cdot)\right]^j \ \rangle \ .
\end{multline}
From (\ref{G2}) and the weak Young  inequality we see that for $0<\al\le 1$, the operator $P\left[ T_{\xi',\eta}- T_{\xi,\eta}\right]$  from $\mathcal{H}_\mathcal{F}^p(\Z^d)$ to  $\mathcal{H}_\mathcal{F}^2(\Z^d)$ with $p=2d/(d+2\al)$ is bounded with norm satisfying
\be \label{C4}
\| P\left[T_{\xi',\eta}- T_{\xi,\eta}\right]\|_{p,2} \ \le \  C|\xi'-\xi|^\al
\ee
for a constant $C$ depending only on $d$ if  $d\ge 3$. In the case $d\le 2$ we need to take $\alpha<d/2$, in which case $C$ depends also on $\al$.  For the inequality (\ref{C4}) to hold it is necessary to include the projection $P$ (see remark following Proposition 2.1).

It follows now from Lemma 4.1 and (\ref{B4}), (\ref{C4}) that
\be \label{D4}
\|h_m(\xi',\eta)-h_m(\xi,\eta)\| \ \le \ C|\xi'-\xi|^\al (1-\la/\La)^{m+1}[1+\del(p)]^{m-1} \ ,
\ee
where $p=2d/(d+2\al)$. Note here we are using the fact  that (\ref{A4}) implies that a column vector of ${\bf b}(\cdot)$ is in $\mathcal{H}_\mathcal{F}^p(\Z^d)$ with norm less than $2\ga/(1+\ga)$. The uniform H\"{o}lder continuity of the family of  functions $q(\cdot,\eta), \ 0<\eta\le\La,$ follows from (\ref{D4}) and Lemma 4.1 by taking $p$ sufficiently close to $2$ so that $(1-\la/\La)[1+\del(p)]<1$.

The uniform H\"{o}lder continuity of the family of  functions $q(\xi,\cdot), \ \xi\in[-\pi,\pi]^d,$ can be obtained in a similar way by observing that 
\be \label{CC4}
\| P\left[T_{\xi,\eta'}- T_{\xi,\eta}\right]\|_{p,2} \ \le \  C|(\eta'-\eta)/\La|^{\al/2},
\ee
where $C$ and $p$ are as in (\ref{C4}).
\end{proof}
For $1\le p\le\infty$ let $L^p(\Z^d,\C^d\otimes\C^d)$ be the Banach space of $d\times d$ matrix valued functions $g:\Z^d\ra\C^d\otimes\C^d$ with norm $\|g\|_p$ defined by
\be \label{E4}
\|g\|_p^p  =  \sup_{v\in\C^d:|v|=1}\sum_{x\in\Z^d} |g(x)v|^p \ {\rm if  \ }p<\infty, \qquad \|g\|_\infty= \sup_{v\in\C^d:|v|=1}\left[ \ \sup_{x\in\Z^d} |g(x)v| \ \right] \ ,
\ee
where $|g(x)v|$ is the Euclidean norm of the vector $g(x)v\in\C^d$. We similarly define spaces  $L^p([-\pi,\pi]^d\times\Om,\C^d\otimes\C^d)$ of $d\times d$ matrix valued functions $g:[-\pi,\pi]^d\times\Om\ra \C^d\otimes\C^d$ with norm $\|g\|_p$ defined by
\begin{multline} \label{F4}
\|g\|_p^p=  \sup_{v\in\C^d:|v|=1}\frac{1}{(2\pi)^d} \int_{[-\pi,\pi]^d} \langle \ |g(\xi,\cdot)v|^2 \ \rangle ^{p/2}  d\xi \quad {\rm if \ }p<\infty, \\
\quad \|g\|_\infty= \sup_{v\in\C^d:|v|=1}\left[ \ \sup_{\xi\in[-\pi,\pi]^d}  \langle \ |g(\xi,\cdot)v|^2 \ \rangle ^{1/2} \ \right]  \ .
\end{multline}
For $\eta>0,m=1,2,..,$  we define an operator $T_{m,\eta}$  from functions $g:\Z^d\ra\C^d\otimes \C^d$ to periodic functions $T_{m,\eta}g:[-\pi,\pi]^d\times\Om\ra \C^d\otimes\C^d$ by
\be \label{G4}
T_{m,\eta}g(\xi,\cdot) \ = \   \sum_{x\in\Z^d} g(x)e^{-ix.\xi}\tau_xP{\bf b}(\cdot)\left[ PT_{\xi,\eta}{\bf b}(\cdot)\right]^{m-1} \ .
\ee
 We shall be interested in showing that  for certain values of $p,q$ the operator $T_{m,\eta}$ is bounded from   $L^p(\Z^d,\C^d\otimes\C^d)$  to  $L^q([-\pi,\pi]^d\times\Om,\C^d\otimes\C^d)$, uniformly in $\eta>0$. In \cite{cn2} this was already shown for $p=1,q=\infty$ and $p=q=2$, with the corresponding operator norms $\|T_{m,\eta}\|_{p,q}$ satisfying the inequalities
\be \label{H4}
\|T_{m,\eta}\|_{1,\infty} \  \le \  (1-\la/\La)^m, \quad \|T_{m,\eta}\|_{2,2} \ \le \  \sqrt{d} \ (m+1)(1-\la/\La)^m \ .
\ee
Observe now from the proof of Proposition 4.1 that for the independent random variable environment corresponding to (\ref{A4}) we can improve upon (\ref{H4}). Thus there exists $p_0(\La/\la)$ with  $1<p_0(\La/\la)<2$ depending only on $d$ and $\La/\la$, such that
\be \label{I4}
\|T_{m,\eta}\|_{p,\infty} \  \le \  (1-\la/\La)^{m/2}, \quad {\rm for \ } 1\le p\le p_0(\La/\la) \ .
\ee
It follows now from (\ref{H4}), (\ref{I4}) and the Riesz convexity theorem \cite{sw} that
\begin{multline} \label{J4}
\|T_{m,\eta}\|_{p,q} \  \le \  \sqrt{d} \ (m+1)(1-\la/\La)^{m/2},  \quad {\rm for \ } 1\le p\le 2, \\
{\rm and \ } 1\ge \frac{1}{p}+\frac{1}{q} \ge 1-
\left[1-\frac{1}{p_0(\La/\la)}\right]\left[1-\frac{2}{q}\right] \ .
\end{multline}
We can use (\ref{J4}) to obtain an improvement on Proposition 2.2 in the case $|m|=1$. 
\begin{lem}
Suppose $d\ge 2$ and $(\Om,\mathcal{F},P)$ in Proposition 2.2.  is the Bernoulli environment corresponding to (\ref{A4}) Then in the case $|m|=1$ the derivative (\ref{X2})  is in the space $L^p([-\pi,\pi]^d)$ with $p=[d+\del(\La/\la)]/|m|$ and its norm is bounded by $\La C(\La/\la)$, for positive constants  $\del(\La/\la)$ and $C(\La/\la)$ depending only on $d$ and $\La/\la\ge 1$.  
\end{lem} 
\begin{proof}
Observe from (\ref{AC2}) and (\ref{M2})   that 
\be \label{K4}
\left( \frac{\pa}{\pa \xi_j} \right)h_m(\xi,\eta)= \sum_{k=1}^m\sum_{r=1}^4 \langle \  \left[T_{m+1-k,\eta}g_{j,r}(\xi,\cdot)\right]^* \  T_{k,\eta}h_{j,r}(\xi,\cdot) \ \rangle \ 
\ee
for  certain $d\times d$ matrix valued functions $g_{j,r}(x), h_{j,r}(x), \ x\in\Z^d$. The functions $g_{j,r}(\cdot), h_{j,r}(\cdot)$ are determined from their Fourier transforms (\ref{Z2})   by the  formula
\be \label{L4}
\sum_{r=1}^4\hat{g}_{j,r}(\zeta)^*\hat{h}_{j,r}(\zeta)  \ = \  -\frac{\pa}{\pa \zeta_j}\left[ \frac{e(-\zeta)e(-\zeta)^*}{\eta/\La+e(-\zeta)^*e(-\zeta)}\right] \ ,
\ee
which follows from (\ref{AC2}).
Evidently one can choose the $g_{j,r}(\cdot), h_{j,r}(\cdot)$ satisfying (\ref{L4}) so that  they also satisfy the inequality
\be \label{M4}
\|g_{j,r}(x)\|+\| h_{j,r}(x)\| \ \le \  \frac{C\exp[-\ga |x| \sqrt{\eta/\La}]}{( |x|+1)^{d-1/2}}, \quad x\in\Z^d, \ 0<\eta\le \La,
\ee
for positive constants $C,\ga$ depending only on $d\ge 1$.  Hence we may estimate the RHS of (\ref{K4}) by using (\ref{J4}) for any $p>d/(d-1/2)$. Since we require $p\le 2$ in (\ref{J4}) it is only possible to do this when $d\ge 2$. The result follows. 
\end{proof}
As in \cite{cn2} we need to obtain norm estimates analogous to (\ref{J4}) on multilinear versions of the operator (\ref{G4}) in order to prove estimates on the derivatives (\ref{X2}) for $|m|>1$. For $\eta>0, \ k\ge 1$ and $m_1,m_2,..,m_k=1,2,..,$  we define a multlinear operator $T_{m_1,m_2,..m_k,\eta}$  from a sequence $[g_1,g_2,..,g_k]$ of $k$ functions $g_j:\Z^d\ra\C^d\otimes \C^d, \ j=1,..,k$, to periodic functions $T_{m_1,m_2,..m_k,\eta}[g_1,g_2,..g_k]:[-\pi,\pi]^d\times\Om\ra \C^d\otimes\C^d$ by
\be \label{N4}
T_{m_1,..,m_k,\eta}[g_1,g_2,...,g_k](\xi,\cdot) \ = \   \sum_{x_1,..x_k \in\Z^d} \prod_{j=1}^k g_j(x_j)e^{-ix_j.\xi}\tau_{x_j}P{\bf b}(\cdot)\left[ PT_{\xi,\eta}{\bf b}(\cdot)\right]^{m_j-1} \ .
\ee
For $p$ satisfying $1\le p\le \infty$  let $p'$ be the conjugate of $p$, so $1/p+1/p'=1$. In \cite{cn2} the following generalization of (\ref{H4}) was obtained:
\begin{lem}
Suppose $2\le q\le\infty$ and $p_1,...,p_k$ with $1\le p_1,...,p_k\le 2$ satisfy the identity 
\be \label{O4}
\frac{1}{p_1'} +\frac{1}{p_2'}+\cdots +\frac{1}{p_k'} \ = \ \frac{1}{q} \ .
\ee
If for $j=1,..,k$, the function $g_j\in L^{p_j}(\Z^d,\C^d\otimes\C^d)$, then  $T_{m_1,m_2,..m_k,\eta}[g_1,g_2,..g_k]$ is in $L^q([-\pi,\pi]^d\times\Om,\C^d\otimes\C^d )$ and
\be \label{P4}
\| \ T_{m_1,m_2,..m_k,\eta}[g_1,g_2,..g_k] \ \|_q \  \le \  d^{k/2}(m+1)\left(1-\la/\La\right)^m \prod_{j=1}^k \|g_j\|_{p_j} \ ,
\ee
where $m=m_1+\cdots +m_k$. 
\end{lem}
 For the independent random variable environment corresponding to (\ref{A4}) one can obtain an improvement of Lemma 4.3 analogous to the improvement  (\ref{I4}) over (\ref{H4}).
 \begin{lem}
 Suppose $(\Om,\mathcal{F},P)$ is the independent random variable environment corresponding to (\ref{A4}).  Then there exists $p_0(\La/\la)$ with $1<p_0(\La/\la)\le 2$ depending only on $\La/\la$ and $d$  such that if $2\le q\le  \infty$ and 
 \be \label{Q4}
\frac{1}{q} \ \le \  \frac{1}{p_1'} +\frac{1}{p_2'}+\cdots +\frac{1}{p_k'} \ \le \ \frac{1}{q} + \left[1-\frac{1}{p_0(\La/\la)}\right]\left[1-\frac{2}{q}\right]\ ,
 \ee
 the function $T_{m_1,m_2,..m_k,\eta}[g_1,g_2,..g_k]$ is in $L^q([-\pi,\pi]^d\times\Om,\C^d\otimes\C^d)$ and
\be \label{R4}
\| \ T_{m_1,m_2,..m_k,\eta}[g_1,g_2,..g_k] \ \|_q \  \le \ d^{k/2}(m+1) \left(1-\la/\La\right)^{m/2} \prod_{j=1}^k \|g_j\|_{p_j} \ .
\ee 
 \end{lem}
 \begin{proof}
 We have already proved the lemma for $k=1$ so we consider the case $k=2$. Observe that  (\ref{R4}) holds if $p_2=1$  in a similar way to the proof of  (\ref{J4}). Evidently Lemma 4.3 implies that (\ref{R4}) also holds if  $q=2$. Hence by an application of the  Riesz convexity theorem we conclude that the result holds for the case $k=2$. To prove the result for $k=3$ we proceed similarly, using the fact that we have proved it for $k=2$, and Lemma 4.3 with $k=3$ and $q=2$.
 \end{proof}
 \begin{proof}[Proof of Theorem 3.1] 
 We argue just as in \cite{cn2} to show that for $a=0$, Lemma 4.4 implies the derivative (\ref{X2}) is in the space $L^p([-\pi,\pi]^d)$ with $p=d/(|m|-\al)$, and hence in $L^p_w([-\pi,\pi]^d)$. A similar argument holds for the difference (\ref{Y2}).  Since one can easily show that the proofs of Proposition 4.1 and Lemma 4.4 continue to hold for $|a|\le C_1\sqrt{\la\eta/\La^2}$, we have proven Theorem 3.1 for the environment corresponding to (\ref{A4}). It is shown in \cite{cn1} how to extend the argument for the Bernoulli environment corresponding to (\ref{A4}) to general i.i.d. environments  ${\bf a}(\tau_x\cdot), \ x\in\Z^d$. We have therefore proven Theorem 3.1 for ${\bf a}(\tau_x\cdot), \ x\in\Z^d$, i.i.d. such that  (\ref{A1}) holds.
 \end{proof}
\vspace{.2in}
\section{Massive Field Theory Environment}
In this section we shall show that Theorem 3.1 holds if $(\Om,\mathcal{F},P)$ is given by the massive field theory environment determined by (\ref{L1}). The main tool we use to prove the theorem is the Brascamp-Lieb (BL) inequality \cite{bl}. This is perhaps natural to expect since the BL inequality is needed to prove that the operators $\tau_{{\bf e}_j}, \ 1\le j\le d$, on $\Om$ are strong mixing, which by Proposition 2.1 implies the continuity of the function $q(\xi,\eta)$ in the region $\xi\in\R^d, \ \eta\ge 0$.   

We recall the main features of the construction of the measure (\ref{L1}). 
Let $L$ be a positive even integer and $Q=Q_L\subset \Z^d$ be the integer lattice points in the cube centered at the origin with side of length $L$. By a periodic function $\phi:Q\ra\R$ we mean a function $\phi$ on $Q$ with the property that $\phi(x)=\phi(y)$ for all $x,y\in Q$ such that $x-y=L{\bf e}_k$ for some $k, \ 1\le k\le d$. Let $\Om_Q$ be the space of all periodic functions $\phi:Q\ra\R$, whence $\Om_Q$ with $Q=Q_L$ can be identified with $\R^{N}$ where $N=L^d$. Let $\mathcal{F}_Q$ be the Borel algebra for $\Om_Q$ which is generated by the open sets of $\R^{N}$. For $m>0$, we define a probability measure $P_{Q}$ on $(\Om_Q,\mathcal{F}_Q)$ as follows:
\begin{multline} \label{A5}
<F(\cdot)>_{\Om_Q}= \\
 \int_{\R^N} F(\phi(\cdot))\exp\left[-\sum_{x\in Q} \left\{V(\nabla \phi(x))+\frac{1}{2}m^2\phi(x)^2\right\}\right]  \prod_{x\in Q} d\phi(x)/{\rm normalization} \ ,
 \end{multline}
 where $F:\R^N\ra\R$ is a continuous function such that $|F(z)|\le C\exp[A|z|], \ z\in\R^N$, for some constants $C,A$.  Differentiating the probability density in (\ref{A5}) we see that for any $f\in\Om_Q$
 \be \label{B5}
  \langle \   \left[\big(\na f(\cdot),V'(\na\phi(\cdot))\big)+m^2 \big(f(\cdot),\phi(\cdot)\big)\right]  \  \rangle_{\Om_Q}  \ = \ 0,
\ee
where $(\cdot,\cdot)$ denotes the Euclidean inner product on $L^2(Q)$. 
Hence by translation invariance of the measure (\ref{A5}) we conclude that  $\langle(f,\phi)\rangle_{\Om_Q}= 0$ for all $f(\cdot)$. The BL inequality  \cite{bl}  applied to (\ref{A5}) and function $F(\phi(\cdot))=\exp[(f,\phi)]$  then yields the inequality 
 \be \label{C5}
 \langle  \exp[(f,\phi)]\rangle_{\Om_Q} \ \  \le  \ \  \exp\left[\frac{1}{2} (f, \{-\la\Del+m^2\}^{-1} f)\right] \ .
 \ee
 The probability space $(\Om,\mathcal{F},P)$ on  fields $\phi:\Z^d\ra\R$ is obtained as the limit of the spaces $(\Om_Q,\mathcal{F}_Q,P_{Q})$ as $|Q|\ra\infty$. In particular one has from Lemma 2.2 of  \cite{c1} the following result:
\begin{proposition} Assume $m>0$ and let $F:\R^{k}\ra\R$ be a $C^1$ function which satisfies the inequality
\be \label{D5}
            |DF(z)|\le A\exp[ \ B|z| \ ], \quad z\in\R^{k},
 \ee
 for some constants $A,B$. Then for any $x_1,....x_k\in\Z^d$, the limit
 \be \label{E5}
 \lim_{|Q|\ra\infty} \langle F\left(\phi(x_1),\phi(x_2),.....,\phi(x_k)\right)\rangle_{\Om_Q}= \langle F\left(\phi(x_1),\phi(x_2),.....,\phi(x_k)\right)\rangle
\ee
exists and is finite.
\end{proposition}
From (\ref{C5})  and the Helly-Bray theorem \cite{br,d} one sees that Proposition 5.1 implies the existence of a unique Borel probability measure on $\R^m$ corresponding to the probability distribution of the variables $(\phi(x_1),..,\phi(x_m))\in\R^m$, and this measure satisfies (\ref{E5}). The Kolmogorov construction \cite{br,d} then implies the existence of a Borel measure on fields $\phi:\Z^d\ra\R$ with finite dimensional distribution functions satisfying (\ref{E5}). This is the measure (\ref{L1}), which we have formally written as having a density with respect to Lebesgue measure. Note however that we do not have a proof of this fact.  In particular, we do not know if the distribution measure for the one dimensional variable $\phi(x)\in\R$ is absolutely continuous with respect to Lebesgue measure.
\begin{proposition}
Let  $(\Om,\mathcal{F},P)$ be the probability space corresponding to the massive field theory  with measure (\ref{L1}). Then the operators  $\tau_{{\bf e}_j}, \ 1\le j\le d$, on $\Om$ are strong mixing.
\end{proposition}
\begin{proof}
It will be sufficient for us to show \cite{p} that for any $m\ge 1$ and $x_1,..,x_m\in\Z^d$,
\begin{multline} \label{F5}
\lim_{n\ra\infty} \langle \ f(\phi(x_1+n{\bf e}_1),....,\phi(x_m+n{\bf e}_1)) \  g(\phi(x_1),....,\phi(x_m)) \ \rangle \ =  \\
\langle \  f(\phi(x_1),....,\phi(x_m)) \ \rangle \  \langle \  g(\phi(x_1),....,\phi(x_m)) \ \rangle 
\end{multline}
for all $C^\infty$ functions $f,g:\R^m\ra\R$ with compact support.  We shall just consider the case $m=1$ since the general case follows from this in a straightforward manner.  

We define the function $h:\Z\ra\R$ by
\be \label{G5}
h(n)= \langle \ f(\phi(n{\bf e}_1))  \ g(\phi(0)) \ \rangle - \langle  \ f(\phi(0)) \ \rangle  \   \langle  \ g(\phi(0)) \ \rangle , \quad n\in\Z.
\ee
Then Proposition 5.1 implies that for any function $k:\Z\ra\R$ of  finite support,
\be \label{H5}
\lim_{|Q|\ra\infty} \sum_{n\in \Z} k(n)h_Q(n) \ = \ \sum_{n\in\Z} k(n)h(n) \ ,
\ee
where $h_Q(\cdot)$ is given by
\be \label{I5}
h_Q(n) \ = \  \langle \ f(\phi(n{\bf e}_1))  \ g(\phi(0)) \ \rangle_{\Om_Q} - \langle  \ f(\phi(0)) \ \rangle_{\Om_Q}  \   \langle  \ g(\phi(0)) \ \rangle_{\Om_Q} , \quad n\in \Z.
\ee
We assume that $Q=Q_L$ with $L$ large enough  so  that the support of $k(\cdot)$ is contained in the interval $[-L/2+1,L/2-1]$. Hence  both $k(\cdot)$ and  $h_Q(\cdot)$ are periodic functions on $I_L=\Z\cap[-L/2,L/2]$.  We may therefore write the sum on the LHS of (\ref{H5}) in its Fourier representation. Thus the Fourier transform of a periodic function $F:I_L\ra\C$ is the periodic function $\hat{F}:\hat{I}_L\ra\C$ defined by
\be \label{J5}
\hat{F}(\zeta) \ = \ \sum_{x\in I_L} F(x)e^{ix\zeta}, \quad \zeta\in \hat{I}_L,
\ee
where $\hat{I}_L$ is the set of lattice points of $(2\pi/L)\Z$ which lie in the interval $[-\pi,\pi]$.
Then
\be \label{K5}
 \sum_{n\in \Z} k(n)h_Q(n) \ = \  \frac{1}{2\pi}\int_{\hat{I}_L} \hat{k}(\zeta) \ \overline{\hat{h}_Q(\zeta)} \ d\zeta \ ,
\ee
with integration on $\hat{I}_L$ defined by
\be \label{L5}
\int_{\hat{I}_L} \ = \ \frac{2\pi}{L}\sum_{\zeta\in\hat{I}_L} \ .
\ee
We can estimate $\hat{h}_Q(\zeta)$ by using translation invariance of the measure (\ref{A5}) and the BL inequality. Thus translation invariance implies that
\be \label{M5}
\hat{h}_Q(\zeta) \ = \ \frac{1}{L}\langle \ a(f,\zeta, \phi(\cdot)) \ \overline{a(g,\zeta,\phi(\cdot))} \ \rangle_{\Om_Q} \ ,
\ee
where $a(f,\zeta, \phi(\cdot))$ is given by the formula
\be \label{N5}
a(f,\zeta, \phi(\cdot)) \ = \ \sum_{n\in I_L} [f(\phi(n{\bf e}_1))-\langle \ f(\phi(n{\bf e}_1)) \ \rangle_{\Om_Q}] \  e^{in\zeta} \ .
\ee
 The BL inequality implies then that
\be \label{O5}
\langle \ |a(f,\zeta, \phi(\cdot))|^2 \ \rangle_{\Om_Q} \ \le  \  \frac{L\|Df(\cdot)\|_\infty^2}{m^2} \ .
\ee
Hence there is a constant $C$ independent of $Q$ such that $|\hat{h}_Q(\zeta)|\le C, \ \zeta\in\hat{I}_L$. Applying then the Schwarz inequality in (\ref{K5}) and using (\ref{H5}) we conclude that
\be \label{P5}
\big| \ \sum_{n\in\Z} k(n)h(n) \  \big| \ \le \  C\left\{\sum_{n\in\Z} k(n)^2 \right\}^{1/2}
\ .
\ee 
It follows that $h(\cdot)\in L^2(\Z)$, and consequently $\lim_{n\ra\infty} h(n)=0$.
\end{proof}
We shall show how the BL inequality can be used to improve the most elementary of the inequalities contained in $\S 2$. Thus let us consider an equation which differs from (\ref{D2}) only  in that the projection operator $P$ has been omitted, 
\be \label{Q5}
\eta\Phi(\xi,\eta,\om)+\pa_\xi^*{\bf a}(\om)\pa_\xi\Phi(\xi,\eta,\om)=-\pa^*_\xi {\bf a}(\om), \quad \eta>0, \ \xi\in \R^d, \ \om\in\Om.
\ee
 For any $v\in\C^d$ we multiply the row vector (\ref{Q5}) on the right by the column vector $v$ and by the function $\overline{\Phi(\xi,\eta,\om)v}$ on the left.  Taking the expectation we see  that 
\be \label{R5}
\|P\pa_\xi\Phi(\xi,\eta,\cdot)v\| \ \le \  \|\pa_\xi\Phi(\xi,\eta,\cdot)v\|  \ \le \  \frac{\La|v|}{\la}  \ .
\ee
where $\|\cdot\|$ denotes the norm in $\mathcal{H}(\Om)$. Let  $g:\Z^d\ra\C^d\otimes\C^d$ be in $L^p(\Z^d,\C^d\otimes\C^d)$ with norm given by  (\ref{E4}). If $p=1$ then (\ref{R5}) implies that
\be \label{S5}
\|P\sum_{x\in\Z^d} g(x)\pa_\xi\Phi(\xi,\eta,\tau_x\cdot)v\|  \ \le \  \|\sum_{x\in\Z^d} g(x)\pa_\xi\Phi(\xi,\eta,\tau_x\cdot)v\|  \ \le \ \frac{\La|v|}{\la} \|g\|_1 \ .
\ee
The BL inequality enables us to improve (\ref{S5}) to allow $g\in L^p(\Z^d,\C^d\otimes\C^d)$ for some $p>1$.
\begin{proposition}
Suppose ${\bf a}(\cdot)$ in (\ref{Q5}) is as in the statement of Theorem 1.2.  Then there exists $p_0(\La/\la)$ depending only on $d$ and $\La/\la$ and satisfying $1<p_0(\La/\la)< 2$, such that for $g\in L^p(\Z^d,\C^d\otimes\C^d)$ with $1\le p\le p_0(\La/\la)$ and $v\in\C^d$,
\be \label{T5}
\|P\sum_{x\in\Z^d} g(x)\pa_\xi\Phi(\xi,\eta,\tau_x\cdot)v\|   \ \le \ \frac{\La_1C|v|}{m\La} \|g\|_p \ ,
\ee
where $\La_1$ is the constant in Theorem 1.2 and  $C$ depends only on $d$ and $\La/\la$. 
\end{proposition}
\begin{proof}
As in Proposition 5.1 we first assume that $g(\cdot)$ has finite support in $\Z^d$. For a cube $Q$ containing the support of $g(\cdot)$  let $\Phi_Q(\xi,\eta,\cdot)$ be the solution to (\ref{Q5}) with ${\bf a}(\phi)= \tilde{\bf a}(\phi(0)), \ \phi\in\Om_Q$, so the random environment for (\ref{Q5}) is $(\Om_Q,\mathcal{F}_Q,P_Q)$. The BL inequality implies that
\be \label{U5}
\|P\sum_{x\in\Z^d} g(x)\pa_\xi\Phi_Q(\xi,\eta,\tau_x\cdot)v\|^2   \ \le \ 
\frac{1}{m^2}\sum_{z\in Q} \| \ \frac{\pa}{\pa \phi(z)}  \ \sum_{x\in\Z^d} g(x)\pa_\xi\Phi_Q(\xi,\eta,\tau_x\cdot)v\|^2 \ . 
\ee

Translation operators $\tau_x, \ x\in\Z^d$, act on functions $F_Q:\Om_Q\ra\C$  by $\tau_x F_Q(\phi(\cdot))=F_Q(\tau_x\phi(\cdot))$. We shall also need to use translation operators $T_x, \ x\in\Z^d$, which act on functions $G_Q:Q\times \Om_Q\ra\C$ by $T_x G_Q(z,\phi(\cdot))=G_Q(z+x,\phi(\cdot))$, so $T_x$ acts on the first variable of $G_Q(\cdot,\phi(\cdot))$. The operators  $\tau_x, \ x\in\Z^d$, act on the second variable of $G_Q(\cdot,\phi(\cdot))$, and it is clear that they commute with the $T_x, \ x\in\Z^d$.
For a $C^1$ function $F_Q:\Om_Q\ra\C$ let  $dF_Q(\cdot,\phi(\cdot))$ denote its gradient so that 
\be \label{V5}
dF_Q(z,\phi(\cdot)) \ = \ \frac{\pa}{\pa \phi(z)} F_Q(\phi(\cdot)) \ , \quad z\in Q.
\ee
One easily sees that 
\be \label{W5}
d[\tau_x F_Q] \ = \ T_{-x}\tau_x dF_Q, \quad x\in \Z^d,
\ee
whence it follows from (\ref{C2})  that
\be \label{X5}
d[\pa_{j,\xi}\tau_x F_Q] \ = \ [e^{-i{\bf e}_j.\xi}T_{-{\bf e}_j}\tau_{{\bf e}_j}-1]T_{-x}\tau_x dF_Q, \qquad 1\le j\le d, \ x\in \Z^d.
\ee
Hence if we define a function $G_Q:Q\times \Om_Q\ra\C$ by
\be \label{Y5}
G_Q(y,\phi(\cdot)) \ = \ dF_Q(-y,\tau_y\phi(\cdot)), \quad y\in Q,
\ee
then (\ref{X5}) implies that
\be \label{Z5}
d[\pa_{j,\xi}\tau_x F_Q](z,\phi(\cdot)) \ = \ \na_{j,\xi}G_Q(x-z,\tau_z\phi(\cdot)),  \qquad 1\le j\le d, \ x,z\in \Z^d,
\ee
where $\na_{\xi}=(\na_{1,\xi},..,\na_{d,\xi})$ and its adjoint  $\na^*_{\xi}=(\na^*_{1,\xi},..,\na^*_{d,\xi})$ are generalizations of the gradient operators (\ref{C1}),
\begin{eqnarray} \label{AA5}
  \na_{j,\xi} \phi(x) &=& e^{-i{\bf e}_j.\xi}\phi (x + {\bf e}_j) - \phi(x),  \qquad 1\le j\le d, \ x\in\Z^d, \\
  \na^*_{j,\xi} \phi(x) &=& e^{i{\bf e}_j.\xi}\phi (x - {\bf e}_j) - \phi(x), \qquad 1\le j\le d, \ x\in\Z^d, \nonumber
\end{eqnarray}
and act on the first variable of $G_Q(\cdot,\phi(\cdot))$.
On taking $F_Q(\phi(\cdot))=\Phi_Q(\xi,\eta,\phi(\cdot))v$ and defining $G_Q$ by (\ref{Y5}), we conclude from (\ref{Z5}) that (\ref{U5}) is the same as
\be \label{AB5}
\|P\sum_{x\in\Z^d} g(x)\pa_\xi\Phi_Q(\xi,\eta,\tau_x\cdot)v\|^2   \ \le \ 
\frac{1}{m^2}\sum_{z\in Q} \|  \ \sum_{x\in\Z^d} g(x)\na_{\xi}G_Q(x-z,\phi(\cdot))\|^2 \ . 
\ee

We can find an equation for $G_Q(y,\phi(\cdot))$ by applying the gradient operator $d$ to (\ref{Q5}). Thus from (\ref{X5}) we obtain the equation
\begin{multline} \label{AC5}
\eta \  dF_Q(\cdot,\phi(\cdot))+D_\xi^*\tilde{{\bf a}}(\phi(0))D_\xi  \ dF_Q(\cdot,\phi(\cdot)) \\
 = \ -D_\xi^*[ \ \del(\cdot) D\tilde{{\bf a}}(\phi(0))\{v+\pa_\xi F_Q(\phi(\cdot))\} ] \ ,
\end{multline}
where the operators  $D_{\xi}=(D_{1,\xi},..,D_{d,\xi})$ and $D^*_{\xi}=(D^*_{1,\xi},..,D^*_{d,\xi})$ are given by the formulae
\be \label{AD5}
D_{j,\xi}=[e^{-i{\bf e}_j.\xi}T_{-{\bf e}_j}\tau_{{\bf e}_j}-1], \quad 
D^*_{j,\xi}=[e^{i{\bf e}_j.\xi}T_{{\bf e}_j}\tau_{-{\bf e}_j}-1] \ , \quad 1\le j\le d,
\ee
and $\del:Q\ra\R$ is the Kronecker delta function, $\del(0)=1, \ \del(z)=0, \ z\ne 0$. Evidently for any $y\in\Z^d$ we can replace $\phi(\cdot)$ in (\ref{AC5}) by $\tau_y\phi(\cdot)$. If we now evaluate (\ref{AC5}) with $\tau_y\phi(\cdot)$ substituted for $\phi(\cdot)$ and with the first variable equal to $-y$ we obtain an equation for the function $G_Q(\cdot,\phi(\cdot))$ of (\ref{Y5}),
\begin{multline} \label{AE5}
\eta \  G_Q(y,\phi(\cdot))+\na_\xi^*\tilde{{\bf a}}(\phi(y))\na_\xi  \ G_Q(y,\phi(\cdot)) \\
 = \ -\na_\xi^*[ \ \del(-y) D\tilde{{\bf a}}(\phi(y))\{v+\pa_\xi F_Q(\tau_y\phi(\cdot))\} ] \ , \quad y\in Q, \ \phi(\cdot)\in\Om_Q.
\end{multline}
From (\ref{R5}) and (\ref{AE5}) we immediately see that
\be \label{AF5}
\left(\sum_{y\in Q} \|\na_\xi  \ G_Q(y,\phi(\cdot))\|^2\right)^{1/2} \ \le \ \frac{\La_1}{\la} \|v+\pa_\xi F_Q(\phi(\cdot))\| \  \le \ \frac{\La_1}{\la}\left[1+\frac{\La}{\la}\right]|v|,
\ee
where $\La_1$ is the constant of Theorem 1.2. Hence (\ref{AB5}) implies the inequality
\be \label{AG5}
\|P\sum_{x\in\Z^d} g(x)\pa_\xi\Phi_Q(\xi,\eta,\tau_x\cdot)v\|   \ \le \ 
\frac{\La_1|v|}{\la m}\left[1+\frac{\La}{\la}\right]\|g\|_1 \ . 
\ee

The inequality (\ref{AG5}) is evidently weaker than (\ref{S5}) since it involves the bound $\La_1$ on the derivative of $\tilde{{\bf a}}(\cdot)$. The point is that the method applies by using Meyer's theorem to give bounds in terms of the $p$ norm of $g(\cdot)$ for some $p>1$. To see this first observe that one can take the limit $Q\ra\Z^d$ in the inequality (\ref{AB5}), whence we have that
\be \label{AH5}
\|P\sum_{x\in\Z^d} g(x)\pa_\xi\Phi(\xi,\eta,\tau_x\cdot)v\|^2   \ \le \ 
\frac{1}{m^2}\sum_{z\in \Z^d} \|  \ \sum_{x\in\Z^d} g(x)\na_{\xi}G(x-z,\phi(\cdot))\|^2 \ , 
\ee
where $G(\cdot,\phi(\cdot))$ is the solution to the equation
\begin{multline} \label{AI5}
\eta \  G(y,\phi(\cdot))+\na_\xi^*\tilde{{\bf a}}(\phi(y))\na_\xi  \ G(y,\phi(\cdot)) \\
 = \ -\na_\xi^*[ \ \del(-y) D\tilde{{\bf a}}(\phi(y))\{v+\pa_\xi F(\tau_y\phi(\cdot))\} ] \ , \quad y\in \Z^d, \ \phi(\cdot)\in\Om,
\end{multline}
with $F(\phi(\cdot))=\Phi(\xi,\eta,\phi(\cdot))v$.
To see that the LHS of (\ref{AB5}) converges as $Q\ra\Z^d$ to the LHS of (\ref{AH5}) we expand $\pa_\xi\Phi_Q(\xi,\eta,\cdot)$ in an $L^2$ convergent Neumann series as was done with the solution of (\ref{I2}).  Since $\eta>0$ each term in the corresponding expansion of the expectation on the LHS of (\ref{AB5}) converges by Proposition 5.1 to the same term in the expansion of the LHS of (\ref{AH5}). The tail of the expansion is  uniformly small as $Q\ra\Z^d$ from $L^2$ estimates. A similar argument shows that the RHS of (\ref{AB5}) converges as $Q\ra\Z^d$ to the RHS of (\ref{AH5}).

The Neumann series for the solution to (\ref{AI5}) is given in terms of an operator $T_{\xi,\eta}$  acting on functions $g:\Z^d\times\Om\ra\C^d$ which is analogous to the operator (\ref{G2}),
\be \label{AJ5}
T_{\xi,\eta} g(z,\phi(\cdot)) \ = \  \sum_{x\in \Z^d} \left\{\nabla\nabla^* G_{\eta/\La}(x)\right\}^*\exp[-ix.\xi]  \ g(x+z,\phi(\cdot)), \quad z\in\Z^d, \ \phi(\cdot)\in\Om.
\ee
Let ${\bf B}:\Z^d\times\Om\ra\C^d\otimes\C^d$ be defined by ${\bf B}(y,\phi(\cdot))=\tilde{{\bf b}}(\phi(y))$ where $\tilde{{\bf a}}(\cdot)=\La[I_d-\tilde{{\bf b}}(\cdot)]$.
Equation (\ref{AI5}) is then equivalent to the equation
\be \label{AK5}
\na_\xi  \ G(\cdot,\phi(\cdot))=T_{\xi,\eta}[{\bf B}(\cdot,\phi(\cdot))\na_\xi  \ G(\cdot,\phi(\cdot))]
-T_{\xi,\eta}h(\cdot,\phi(\cdot))\ ,
\ee
with $h:\Z^d\times\Om\ra\C^d$ given by the formula
\be \label{AL5}
h(y,\phi(\cdot)) \ = \ \frac{1}{\La}\del(-y) D\tilde{{\bf a}}(\phi(y))\{I_d+\pa_\xi \Phi(\xi,\eta,\phi(\cdot))\}v, \quad y\in\Z^d, \ \phi(\cdot)\in\Om.
\ee
Consider now the Hilbert space $\mathcal{H}(\Z^d\times\Om)$ of functions $g:\Z^d\times\Om\ra\C^d$ with norm $\|g\|_2$ given by
\be \label{AM5}
\|g\|_2^2 \ = \ \sum_{y\in\Z^d} \|g(y,\phi(\cdot)\|^2 \ ,
\ee
where $\|g(y,\phi(\cdot))\|$ is the norm of $g(y,\phi(\cdot))\in\mathcal{H}(\Om)$. Evidently the function $h$ of (\ref{AL5}) is in $\mathcal{H}(\Z^d\times\Om)$ and 
\be \label{AN5}
\|h\|_2\le \frac{\La_1}{\La}\left[1+\frac{\La}{\la}\right]|v|.
\ee 
Since $\|T_{\xi,\eta}\|\le 1$ and $\|{\bf B}(\cdot,\phi(\cdot))\|\le 1-\la/\La$, we conclude from (\ref{AN5}) on summing the Neumann series for (\ref{AK5})  that  (\ref{AF5}) holds for $Q\ra\Z^d$.  

We may define for any $q\ge 1$ the Banach space $L^q(\Z^d\times\Om,\C^d)$ of functions $g:\Z^d\times\Om\ra\C^d$ with norm $\|g\|_q$ given by
\be \label{AO5}
\|g\|_q^q \ = \ \sum_{y\in\Z^d} \|g(y,\phi(\cdot))\|^q \ .
\ee
As in Lemma 4.1 the operator $T_{\xi,\eta}$ is bounded on $L^q(\Z^d\times\Om,\C^d)$ for $q>1$  with norm  $\|T_{\xi,\eta}\|_q\le 1+\del(q)$, where $\lim_{q\ra 2}\del(q)=0$. Noting that $\|h\|_q$ is bounded by the RHS of (\ref{AN5}) for all $q\ge 1$, we conclude then from (\ref{AK5}) that there exists $q_0(\La/\la)<2$  depending only on $d$ and $\La/\la$ such that
\be \label{AP5}
\|\na_\xi G(\cdot,\phi(\cdot))\|_q \ \le  \ \frac{C_1\La_1|v|}{\La}, \quad q_0(\La/\la)\le q\le 2,
\ee
where the constant $C_1$ depends only on $d$ and $\La/\la$. The result follows from (\ref{AH5}), (\ref{AP5}) and Young's inequality.
\end{proof}
We proceed now to establish Theorem 3.1 for the massive field theory environment $(\Om,\mathcal{F},P)$ along the same lines followed in $\S4$ for the i.i.d. environment.
\begin{lem}
Let $T_{r,\eta}, \ r=1,2,.., \ \eta>0,$ be the operator (\ref{G4}). Then there exists $p_0(\La/\la)$ with $1<p_0(\La/\la)<2$ depending only on $d$ and $\La/\la$, such that
\be \label{AQ5}
\|T_{r,\eta}\|_{p,\infty} \ \le \ \frac{\La_1r}{m\La}(1-\la/\La)^{r/2} \quad {\rm \ for \ } 1\le p\le p_0(\La/\la) \ .
\ee
\end{lem}
\begin{proof}
As in Proposition 5.3  it will be sufficient to assume $g:\Z^d\ra\C^d\otimes\C^d$ has finite support and take the $Q\ra\Z^d$ limit in the BL inequality for  finite cubes $Q\subset\Z^d$. The inequality (\ref{AQ5}) follows immediately from BL in the case $r=1$. Thus we have for $v\in\C^d$ that the norm of $T_{1,\eta}g(\xi,\cdot)v\in\mathcal{H}(\Om)$ is bounded as 
\be \label{AR5}
\|T_{1,\eta}g(\xi,\cdot)v\|^2 \ \le \  \frac{1}{m^2}\sum_{x\in\Z^d} \|g(x)D\tilde{\bf b}(\phi(0))v\|^2 
 \le \ \left(\frac{\La_1}{m\La}\|g(\cdot)\|_p|v|\right)^2 
\ee
for any $p$ satisfying $1\le p\le 2$. 

For $r>1$ we write
\be \label{AS5}
T_{r,\eta}g(\xi,\phi(\cdot))v \ = \ P\sum_{x\in\Z^d} g(x) e^{-ix.\xi}\tau_x {\bf b}(\cdot)\pa_\xi F_r(\phi(\cdot)) \ ,
\quad \phi(\cdot)\in\Om,
\ee
where the functions $F_r(\phi(\cdot))$ are defined inductively by
\begin{eqnarray} \label{AT5}
\frac{\eta}{\La}F_r(\phi(\cdot))+\pa_\xi^*\pa_\xi F_r(\phi(\cdot))&=&P\pa^*_\xi [\tilde{{\bf b}}(\phi(0))\pa_\xi F_{r-1}(\phi(\cdot))],  \ r> 2,  \\
\frac{\eta}{\La}F_2(\phi(\cdot))+\pa_\xi^*\pa_\xi F_2(\phi(\cdot))&=&P\pa^*_\xi [\tilde{{\bf b}}(\phi(0))v]  \ .
\nonumber
\end{eqnarray}
Similarly to (\ref{AI5}) we define for $r\ge 2$ functions $G_r:\Z^d\times\Om\ra\C$ by
\be \label{AU5}
G_r(y,\phi(\cdot)) \ = \ dF_r(-y,\tau_y\phi(\cdot)), \quad y\in \Z^d, \ \phi(\cdot)\in\Om.
\ee
Then from (\ref{AT5}) we see that the $G_r(y,\phi(\cdot)) $ satisfy the equations
\begin{multline} \label{AV5}
\frac{\eta}{\La}G_r(y,\phi(\cdot))+\na_\xi^*\na_\xi G_r(y,\phi(\cdot)) \ =\\
P\na^*_\xi [\del(-y)D\tilde{{\bf b}}(\phi(y))\pa_\xi F_{r-1}(\tau_y\phi(\cdot))+\tilde{{\bf b}}(\phi(y))\na_\xi G_{r-1}(y,\phi(\cdot))],  \ r> 2, \\
\frac{\eta}{\La}G_2(y,\phi(\cdot))+\na_\xi^*\na_\xi G_2(y,\phi(\cdot)) \ = \ P\na^*_\xi [\del(-y)D\tilde{{\bf b}}(\phi(y))v]  \ . 
\end{multline}

From (\ref{AS5}) and BL we have that
\begin{multline} \label{AW5}
\|T_{r,\eta}g(\xi,\cdot)v\|^2 \ \le \\
 \frac{1}{m^2}\sum_{z\in\Z^d} \| \ g(z) e^{-iz.\xi}D\tilde{{\bf b}}(\phi(z))\pa_\xi F_r(\tau_z\phi(\cdot))+\sum_{x\in\Z^d} g(x) e^{-ix.\xi}\tilde{ {\bf b}}(\phi(x))\na_\xi G_r(x-z,\tau_z\phi(\cdot))  \ \|^2 \\
 =
 \frac{1}{m^2}\sum_{z\in\Z^d} \| \ g(z) e^{-iz.\xi}D\tilde{{\bf b}}(\phi(0))\pa_\xi F_r(\phi(\cdot))+\sum_{x\in\Z^d} g(x) e^{-ix.\xi}\tilde{ {\bf b}}(\phi(x-z))\na_\xi G_r(x-z,\phi(\cdot))  \ \|^2 \ . 
\end{multline}
Just as in (\ref{AR5}) we see that
\be \label{AX5}
\frac{1}{m^2}\sum_{z\in\Z^d} \| \ g(z) e^{-iz.\xi}D\tilde{{\bf b}}(\phi(0))\pa_\xi F_r(\phi(\cdot)) \ \|^2 
\le  \ \left\{\frac{\La_1}{m\La}\left(1-\frac{\la}{\La}\right)^{r-1}\|g(\cdot)\|_p|v|\right\}^2 
\ee
for any $p$ satisfying $1\le p\le 2$. The second term in the last expression in (\ref{AW5}) can be bounded using an inequality similar to  (\ref{AP5}).  It is clear from (\ref{AV5}) that
\be \label{AY5}
\|\na_\xi G_r(\cdot,\phi(\cdot))\|_2 \ \le  \  \frac{\La_1}{\La}(r-1)\left(1-\frac{\la}{\La}\right)^{r-1}|v| \ .
\ee
Applying the Calderon-Zygmund theorem \cite{stein} to (\ref{AV5}) we see that there exists $q_0(\La/\la)<2$  depending only on $d$ and $\La/\la$ such that
\be \label{AZ5}
\|\na_\xi G_r(\cdot,\phi(\cdot))\|_q \ \le  \  \frac{\La_1}{\La}(r-1)\left(1-\frac{\la}{\La}\right)^{(r-1)/2}|v| \ ,
\ee
provided $q_0(\La/\la)\le q\le 2$. The result follows from (\ref{AW5}), (\ref{AX5}), (\ref{AZ5})  and Young's inequality. 
\end{proof}
\begin{corollary}
The function $q(\xi,\eta)$  of (\ref{E2}) with domain $\xi\in[-\pi,\pi]^d, \ 0<\eta\le\La$,  is uniformly H\"{o}lder continuous. That is there exist positive constants $C,\alpha$  with $0<\alpha\le 1$ depending only on $d$ and $\La/\la$, such that 
\be \label{BA5}
\|q(\xi',\eta')-q(\xi,\eta)\| \ \le  \ \frac{C\La_1}{m} \left[ \ |\xi'-\xi|^\alpha +|(\eta'-\eta)/\La|^{\al/2} \ \right]
\ee
 for  $\xi',\xi\in[-\pi,\pi]^d$ and $0<\eta,\eta' \le \La$.
\end{corollary} 
\begin{proof}
We proceed as in the proof of  Proposition 4.1. Instead of (\ref{C4}) we use the fact that 
\be \label{BB5}
P\left[ T_{\xi',\eta}- T_{\xi,\eta}\right]{\bf b}(\cdot)\left[ PT_{\xi,\eta}{\bf b}(\cdot)\right]^j  \ = \ 
T_{j+1,\eta} g(\xi,\cdot)
\ee
where
\be \label{BC5}
g(x) \ = \  \{\na\na^*G_{\eta/\La}(x)\}^*[e^{ix.(\xi-\xi')}-1], \quad x\in\Z^d \ .
\ee
Evidently for $0\le \al\le 1$ one has that $g\in L^p(\Z^d,\C^d\otimes\C^d)$ for any $p>d/(d-\al)$ and 
$ \|g\|_p\le C_p |\xi'-\xi|^\al$ for a constant $C_p$ depending only on $p$ and $d$. The H\"{o}lder continuity of $q(\xi,\eta)$ in $\xi$ follows then from Lemma 5.1. The H\"{o}lder continuity of $q(\xi,\eta)$ in $\eta$ can be obtained in a similar way.
\end{proof}
To complete the proof of Theorem 3.1 for the massive field theory environment we need to prove a version of Lemma 4.4 and also that one can do analytic continuation in the variable $\xi\in\R^d$. The proof of this follows along the same lines as in $\S4$.
\vspace{.2in}
\section{Massless Field Theory Environment}
In this section we shall prove Theorem 3.1 for the massless field theory environment $(\Om,\mathcal{F},P)$ with measure given by the $m\ra 0$ limit of the massive field theory  measure (\ref{L1}). The measure is constructed by means of the following result proved in  \cite{c1}:
\begin{proposition} Let $F:\R^{kd}\ra\R$ be a $C^1$ function which satisfies the inequality
\be \label{A6}
            |DF(z)|\le A\exp[ \ B|z| \ ], \quad z\in\R^{kd},
 \ee
 for some constants $A,B$, and $\langle\cdot\rangle_m$ denote the massive field theory expectation with measure (\ref{L1}). Then for any $x_1,....x_k\in\Z^d$, the limit
 \be \label{B6}
 \lim_{m\ra 0} \langle F\left(\na\phi(x_1),\na\phi(x_2),.....,\na\phi(x_k)\right)\rangle_m= \langle F\left(\om(x_1),\om(x_2),.....,\om(x_k)\right)\rangle
\ee
exists and is finite.
\end{proposition}
As for the massive case, Proposition 6.1 defines a unique Borel probability measure on gradient fields $\om:\Z^d\ra\R^d$ by using the  inequality derived from (\ref{C5}),
\be \label{C6}
 \langle  \exp[(f,\na \phi)]\rangle_m \ \  \le  \ \  \exp\left[ |f|^2/2\la\right] \ 
 \ee
for any function $f:\Z^d\ra\R^d$ of finite support. This can most easily be seen by using a simple identity. For a function $G(\om(\cdot))$ of vector fields $\om:\Z^d\ra\R^d$ we define its gradient $d_\om G(\cdot,\om(\cdot))$ similarly to (\ref{V5}) by 
\be \label{G6}
d_\om G(z,\om(\cdot)) \ = \ \frac{\pa}{\pa \om(z)} G(\om(\cdot)) \ , \quad z\in\Z^d.
\ee
Thus $d_\om G(z,\om(\cdot)), \ z\in\Z^d$, is for fixed $\om(\cdot)$ a vector field from $\Z^d$ to $\R^d$, and hence we may compute its divergence  $\na^*d_\om G(z,\om(\cdot)), \ z\in\Z^d$. Then with $d$ defined as in (\ref{V5}) we have the identity
\be \label{H6}
d G(z,\na\phi(\cdot)) \ = \ \na^*d_\om G(z,\om(\cdot)), \quad z\in\Z^d.
\ee
The inequality (\ref{C6}) follows from (\ref{C5}) and (\ref{H6}) on setting $G(\om(\cdot))=(f,\om(\cdot))$.
\begin{proposition}
Let  $(\Om,\mathcal{F},P)$ be the probability space corresponding to the massless field theory  with measure given by the $m\ra 0$ limit of the massive field theory  measure (\ref{L1}). Then the operators  $\tau_{{\bf e}_j}, \ 1\le j\le d$, on $\Om$ are strong mixing.
\end{proposition}
\begin{proof}
The proof follows the same lines as the proof of Proposition 5.2. Thus for $C^\infty$ functions of compact support $f,g:\R^d\ra\R$ let $h_{Q,m}:\Z\ra\R$ be defined similarly to (\ref{I5}) by
\be \label{D6}
h_{Q,m}(n) \ = \  \langle \ f(\na\phi(n{\bf e}_1))  \ g(\na\phi(0)) \ \rangle_{\Om_Q,m} - \langle  \ f(\na\phi(0)) \ \rangle_{\Om_Q,m}  \   \langle  \ g(\na\phi(0)) \ \rangle_{\Om_Q,m} , \quad n\in \Z,
\ee
where we have included the index $m$ to emphasize the dependence of the measure (\ref{A5}) on $m$. 
Following (\ref{N5}),  let  $a(f,\zeta, \om(\cdot))$ be given by the formula
\be \label{E6}
a(f,\zeta, \om(\cdot)) \ = \ \sum_{n\in I_L} [f(\om(n{\bf e}_1))-\langle \ f(\om(n{\bf e}_1)) \ \rangle_{\Om_Q,m}] \  e^{in\zeta} \ .
\ee
Then it will be sufficient for us to show that
\be \label{F6}
\langle \ |a(f,\zeta, \na\phi(\cdot))|^2 \ \rangle_{\Om_Q,m} \ \le  \  \frac{L\|Df(\cdot)\|_\infty^2}{\la} \ ,
\ee
since the RHS of (\ref{F6}) is independent of $m$. This follows from BL and (\ref{H6}).
\end{proof}
\begin{proposition}
Suppose ${\bf a}(\cdot)$ in (\ref{Q5}) is as in the statement of Theorem 1.3.  Then there exists $p_0(\La/\la)$ depending only on $d$ and $\La/\la$ and satisfying $1<p_0(\La/\la)< 2$, such that for $g\in L^p(\Z^d,\C^d\otimes\C^d)$ with $1\le p\le p_0(\La/\la)$ and $v\in\C^d$,
\be \label{I6}
\|P\sum_{x\in\Z^d} g(x)\pa_\xi\Phi(\xi,\eta,\tau_x\cdot)v\|   \ \le \ \frac{\La_1C|v|}{\La\sqrt{\la}} \|g\|_p \ ,
\ee
where $\La_1$ is the constant in Theorem 1.3 and  $C$ depends only on $d$ and $\La/\la$. 
\end{proposition}
\begin{proof}
We proceed as in the proof of Proposition 5.3. Thus from (\ref{H6}) and BL we see that
\be \label{J6}
\|P\sum_{x\in\Z^d} g(x)\pa_\xi\Phi_Q(\xi,\eta,\tau_x\cdot)v\|^2   \ \le \ 
\frac{1}{\la}\sum_{z\in \Z^d} \| \ \frac{\pa}{\pa \om(z)}  \ \sum_{x\in\Z^d} g(x)\pa_\xi\Phi_Q(\xi,\eta,\tau_x\om(\cdot))v\|^2 \ . 
\ee 
The remainder of the proof is exactly as in Proposition 5.3.
\end{proof}
\begin{proof}[Proof of Theorem 3.1]
This follows the same lines as the proof of Theorem 3.1 in $\S5$. 
\end{proof}
\section{Second Difference Estimates}
In this section we show how the inequality (\ref{M1}) follows from Theorem 3.1 and the Delmotte-Deuschel argument \cite{dd}
Our starting point is the representation (\ref{A2}) for the averaged Green's function $G_{{\bf a},\eta}(x)$. 
We introduce a low momentum cutoff into the integral (\ref{A2}), then transform the remainder into configuration space and use the H\"{o}lder continuity results of \cite{dd} for the second derivatives of  $G_{{\bf a},\eta}(x)$.  Thus let $\chi:\R^d\ra\R$ be a $C^\infty$ function with compact support such that the integral of $\chi(\cdot)$ over $\R^d$ equals $1$. We write
\be \label{A7}
G_{{\bf a},\eta}(x) \ = \ [G_{{\bf a},\eta}(x) - \chi_L*G_{{\bf a},\eta}(x)] +  \chi_L*G_{{\bf a},\eta}(x) \ ,
\ee
where $\chi_L(x)=L^{-d}\chi(x/L), \ x\in\R^d$, and $*$ denotes convolution on $\Z^d$. 
Let $\hat{\chi}_L(\xi), \ \xi\in[-\pi,\pi]^d,$ be the Fourier transform of
$\chi_L$ restricted to the lattice $\Z^d$. Then for $L\ge 1$ and integers $m,n\ge 0$ there are constants $C,C_{m,n}$ such that
\be \label{B7}
|\hat{\chi}_L(0)-1|   \le \  C/L \ , \quad |(\na_\xi)^m \hat{\chi}_L(\xi)|\le C_{m,n}L^m/[1+L|\xi|]^n \quad \xi\in[-\pi,\pi]^d \ .
\ee
We assume that $R<|x|<2R$ and choose $L=R^{1-\del}$ for some $\del>0$. Then from the first inequality of (\ref{B7}) and the H\"{o}lder continuity result of \cite{dd}, one has the inequality
\be \label{C7}
|\na\na G_{{\bf a},\eta}(x) - \na\na\chi_L*G_{{\bf a},\eta}(x)| \ \le \ \frac{C}{\La(|x|+1)^{d+\alpha}} e^{-\gamma\sqrt{\eta/\La} |x|},
\ee
for some positive constants $\ga$ depending only on $\La/\la,d$ and  $C,\al$ depending only on $\La/\la,d,\del$.  

Next we consider the integral
\be \label{D7}
\frac{1}{(2\pi)^d}\int_{[-\pi,\pi]^d} 
\frac{e^{-i\xi.x}e_k(\xi)e_j(\xi)}{\eta+e(\xi)^*q(\xi,\eta)e(\xi)} \hat{\chi}_L(\xi)\ d\xi \ = \ \frac{e^{a.x}}{(2\pi)^d} \int_{[-\pi,\pi]^d} e^{-i\xi.x} f_a(\xi,\eta) \ d\xi, 
\ee
where for $a\in\R^d$ the function $f_a(\xi,\eta)$ is given by the formula
\be \label{E7}
f_a(\xi,\eta) \ = \ \frac{e_k(\xi+ia)e_j(\xi+ia)}{\eta+e(\xi-ia)^*q(\xi+ia,\eta)e(\xi+ia)} \hat{\chi}_L(\xi+ia) \ .
\ee
Observe now from  the second inequality of (\ref{B7}) that  for any integer $n\ge 0$ there are positive constants   $C,C_n$, where $C$ depends only on $d$ and $C_n$ on $d$ and $n$, such that
\begin{eqnarray} \label{F7}
|\hat{\chi}_L(\xi+ia)| \ &\le& \ \frac{C_n}{[1+|L\xi|]^n}  \ , \qquad \xi\in[-\pi,\pi]^d, \ |a|\le 1/L \ , \\
|\hat{\chi}_L(\xi+ia)| \ &\le& \ Ce^{C|a|L} \ , \qquad \xi\in[-\pi,\pi]^d, \ |a|\ge 1/L \ .  \label{G7}
\end{eqnarray}
We choose $|a|=C(\La/\la)\sqrt{\eta/\La}$ as in the proof of Theorem 3.2, where $C(\La/\la)$ depends only on $\La/\la$. Then if $|a|\ge 1/L$, one has from (\ref{G7}) that
\be \label{H7}
\frac{e^{a.x}}{(2\pi)^d} \int_{[-\pi,\pi]^d} | f_a(\xi,\eta)| \ d\xi \ \le \ 
\frac{C}{\La(|x|+1)^{d+1}} e^{-\gamma\sqrt{\eta/\La} |x|},
\ee
for some positive constants $\ga$ depending only on $\La/\la,d$ and  $C$ depending only on $\La/\la,d,\del$. If  $|a|\le 1/L$, we see from (\ref{F7}) that
\be \label{I7}
\frac{e^{a.x}}{(2\pi)^d} \int_{[-\pi,\pi]^d\cap\{|\xi|>1/R^{1-2\del}\}} | f_a(\xi,\eta)| \ d\xi \ \le \ 
\frac{C}{\La(|x|+1)^{d+1}} e^{-\gamma\sqrt{\eta/\La} |x|},
\ee
for some positive constants $\ga$ depending only on $\La/\la,d$ and  $C$ depending only on $\La/\la,d,\del$. 

For $a\in\R^d$ we define $g_a(\xi,\eta)$ similarly to $f_a(\xi,\eta)$ by 
\be \label{J7}
g_a(\xi,\eta) \ = \ \frac{e_k(\xi+ia)e_j(\xi+ia)}{\eta+e(\xi-ia)^*q(0,0)e(\xi+ia)} \hat{\chi}_L(\xi+ia) \ .
\ee
Then Theorem 3.1 implies that for $|a|<1/L$ and $\al$ the H\"{o}lder constant in (\ref{D3}) , 
\be \label{K7}
\frac{e^{a.x}}{(2\pi)^d} \int_{[-\pi,\pi]^d\cap\{|\xi|\le 1/R^{1-2\del}\}} | f_a(\xi,\eta)-g_a(\xi,\eta)| \ d\xi \ \le \ 
\frac{C}{\La(|x|+1)^{(d+\al)(1-2\del)}} e^{-\gamma\sqrt{\eta/\La} |x|},
\ee
for some positive constants $\ga$ depending only on $\La/\la,d$ and  $C$ depending only on $\La/\la,d,\del$.  On choosing $\del>0$ in (\ref{K7}) to satisfy $(d+\al)(1-2\del)>d$, we conclude from (\ref{H7}), (\ref{I7}), (\ref{K7}) that
\be \label{L7}
| \na\na\chi_L*G_{{\bf a},\eta}(x) - \na\na\chi_L* G_{{\bf a}_{\rm hom},\eta}(x) | \ \le  \ 
\frac{C}{\La(|x|+1)^{(d+\al)}} e^{-\gamma\sqrt{\eta/\La} |x|},
\ee
for some positive constants $\ga$ depending only on $\La/\la,d$ and  $C,\al$ depending only on $\La/\la,d,\del$. The inequality (\ref{M1}) follows from (\ref{C7}), (\ref{L7}) upon using the fact that
\be \label{M7}
|\na\na G_{{\bf a}_{\rm hom},\eta}(x) - \na\na\chi_L*G_{{\bf a}_{\rm hom},\eta}(x)| \ \le \ \frac{C}{\La(|x|+1)^{d+\del}} e^{-\gamma\sqrt{\eta/\La} |x|},
\ee
for some positive constants $\ga,C$ depending only on $\La/\la,d$.

 \thanks{ {\bf Acknowledgement:} This research was partially supported by NSF
under grant DMS-0553487.

\end{document}